\documentclass{amsart}
\usepackage{amssymb,amsmath}
\usepackage{stmaryrd}
\usepackage{yhmath}
\usepackage{latexsym}
\usepackage{amsthm}
\usepackage{amscd}
\usepackage{mathrsfs}
\usepackage[all]{xy}
\usepackage{mathtools}
\usepackage{arydshln}
\usepackage{bbm}
\usepackage{bm}
\usepackage{url}

\newcommand{\ad}{\mathrm{ad}}

\newcommand{\Z}{\mathbb{Z}}

\newcommand{\R}{\mathbb{R}}

\newcommand{\C}{\mathbb{C}}

\newcommand{\G}{\mathbf{G}}
\newcommand{\J}{\mathbf{J}}

\DeclareMathOperator{\Nr}{Nr}

\DeclareMathOperator{\depth}{depth}

\DeclareMathOperator{\ord}{ord}

\DeclareMathOperator{\GL}{GL}

\DeclareMathOperator{\SL}{SL}
\DeclareMathOperator{\Sp}{Sp}
\DeclareMathOperator{\Mp}{Mp}

\DeclareMathOperator{\U}{U}
\DeclareMathOperator{\SU}{SU}

\DeclareMathOperator{\id}{id}

\theoremstyle{plain}
\newtheorem{thm}{Theorem}[section]
\newtheorem*{thm*}{Theorem}
\newtheorem{prop}[thm]{Proposition}
\newtheorem{lem}[thm]{Lemma}
\newtheorem{cor}[thm]{Corollary}

\theoremstyle{definition}

\theoremstyle{remark}
\newtheorem{rem}[thm]{Remark}
\newtheorem*{claim*}{Claim}

\SelectTips{cm}{11}

\title{Depth preserving property of the local Langlands correspondence for non-quasi-split unitary groups}
\author{Masao Oi}
\address{Graduate School of Mathematical Sciences, 
the University of Tokyo, 3-8-1 Komaba, Meguro-ku, Tokyo 153-8914, Japan.}
\email{masaooi@ms.u-tokyo.ac.jp}

\begin{document}

\begin{abstract}
In this paper, we extend our result on a depth preserving property of the local Langlands correspondence for quasi-split unitary groups (\cite{Oi:2018a}) to non-quasi-split unitary groups by using the local theta correspondence.
The key ingredients are a depth preserving property of the local theta correspondence proved by Pan and a description of the local theta correspondence via the local Langlands correspondence established by Gan--Ichino.
To combine them, we compare splittings for metaplectic covers of unitary groups constructed by Kudla with those constructed by Pan.
\end{abstract}


\maketitle

\section{Introduction}

Let $F$ be a $p$-adic field.
We consider a connected reductive group $\G$ over $F$.
Then the conjectural \textit{local Langlands correspondence} for $\G$ predicts that there exists a natural map from the set of equivalence classes of irreducible smooth representations of $\G(F)$ to the set of conjugacy classes of $L$-parameters of $\G$.
The local Langlands correspondence for general connected reductive groups has not been known at present.
However, thanks to recent developments based on works of a lot of people, for several groups, the correspondence was established.
In particular, for 
\begin{itemize}
\item
general linear groups (\cite{MR1876802}), 
\item
quasi-split symplectic or orthogonal groups (\cite{MR3135650}),
\item
quasi-split unitary groups (\cite{MR3338302}), and
\item
non-quasi-split unitary groups (\cite{KMSW}),
\end{itemize}
the local Langlands correspondence has been established.

In the case of general linear groups, this correspondence is characterized by the theory of $\varepsilon$-factors and $L$-factors of representations.
In the cases of classical groups listed above, the correspondence is characterized by the theory of endoscopy.
However it is known that, beyond these characterizations, the local Langlands correspondence for these groups satisfies a lot of properties.

One example for such phenomena is a \textit{depth preserving property} of the local Langlands correspondence for general linear groups.
To be more precise, let us recall the notion of the \textit{depth} of representations.
When $\G=\GL_{1}$, the group $\G(F)=F^{\times}$ has a maximal open compact subgroup $\mathcal{O}_{F}^{\times}$ (unit group) and its filtration $\{1+\mathfrak{p}_{F}^{n}\}_{n\in\Z_{>0}}$ (higher unit groups).
As its generalization, for a general connected reductive group $\G$ over $F$, we can define various open compact subgroups (called \textit{parahoric subgroups}) of $\G(F)$ and their filtrations (called \textit{Moy--Prasad filtrations}).
By using these subgroups of $\G(F)$, for each irreducible representation $\pi$ of $\G(F)$, we can define its depth (denoted by $\depth(\pi)$), which expresses how large subgroups having an invariant part in the representation are (see Section \ref{subsec:q-spl} for the definition).
On the other hand, for the inertia subgroup $I_{F}$ of the Weil group $W_{F}$ of $F$, we can define its ramification filtration $\{I_{F}^{\bullet}\}$.
Then, by noting that an $L$-parameter of $\G$ is an admissible homomorphism from $W_{F}\times\SL_{2}(\C)$ to the $L$-group of $\G$, we can define the depth of an $L$-parameter $\phi$ (denoted by $\depth(\phi)$), which measures how deep the ramification of the $L$-parameter $\phi$ is (see Section \ref{subsec:q-spl} for the definition).
Then it is known that the local Langlands correspondence for $\GL_{N}$ preserves the depth.
Namely, we have the following:
\begin{thm}[{\cite[2.3.6]{MR2508720} and \cite[Proposition 4.5]{MR3579297}}]
Let $\pi$ be an irreducible smooth representation of $\GL_{N}(F)$ and $\phi$ the $L$-parameter of $\GL_{N}$ corresponding to $\pi$ under the local Langlands correspondence for $\GL_{N}$.
Then we have
\[
\depth(\pi)=\depth(\phi).
\]
\end{thm}
Note that, when $N=1$, this is nothing but the well-known property of the local class field theory about the correspondence between the higher unit groups $\{1+\mathfrak{p}_{F}^{n}\}_{n\in\Z_{>0}}$ and the ramification filtration $\{I_{F}^{\mathrm{ab},n}\}_{n\in\Z_{>0}}$.

Therefore it is a natural attempt to investigate the relationship between the depth of representations and that of $L$-parameters under the local Langlands correspondence for other classical groups.
In \cite{Oi:2018a}, by extending a method of Ganapathy and Varma (\cite{MR3709003}) based on harmonic analysis on $p$-adic reductive groups, we proved the following depth preserving property for quasi-split unitary groups:

\begin{thm}[{\cite[Theorem 5.6]{Oi:2018a}}]\label{thm:intro}
Let $\G$ be a quasi-split unitary group in $N$ variables over $F$.
We assume that the residual characteristic $p$ of $F$ is greater than $N+1$.
Then, for every irreducible smooth representation $\pi$ of $\G(F)$ and its corresponding $L$-parameter $\phi$ of $\G$, we have
\[
\depth(\pi)=\depth(\phi).
\]
\end{thm}

Our aim in this paper is to extend this result to non-quasi-split unitary groups.
Namely our main result is the following:
\begin{thm}[Theorem \ref{thm:main}]
Let $\G$ be a non-quasi-split unitary group in $N$ variables over $F$.
We assume that the residual characteristic $p$ of $F$ is greater than $N+1$.
Then, for every irreducible smooth representation $\pi$ of $\G(F)$ and its corresponding $L$-parameter $\phi$ of $\G$, we have
\[
\depth(\pi)=\depth(\phi).
\]
\end{thm}

To show this, we utilize the \textit{local theta correspondence}.
Let $V$ (resp.\ $V'$) be an $\epsilon$-hermitian (resp.\ $\epsilon'$-hermitian) space over $F$ such that $\epsilon\epsilon'=-1$, and $\U(V)$ (resp.\ $\U(V')$) the corresponding unitary group over $F$.
Then the local theta correspondence gives a correspondence between representations of $\U(V)$ and those of $\U(V')$.
For an irreducible smooth representation $\pi$ of $\U(V)$, we denote the corresponding representation of $\U(V')$ by $\theta(\pi)$.
Here note that $\theta(\pi)$ is possibly zero and that, if $\theta(\pi)$ is not zero, then it is irreducible and smooth.
The key point in our proof is that the local theta correspondence has the following two properties:
\begin{description}
\item[Pan's depth preserving theorem]
The local theta correspondence preserves the depth of representations (\cite{MR1909608}).
Namely, if $\theta(\pi)$ is not zero, then we have
\[
\depth\bigl(\theta(\pi)\bigr)=\depth(\pi).
\]
\item[Gan--Ichino's description of the $L$-parameter of the local theta lift]
If $\theta(\pi)$ is not zero, then the $L$-parameter $\theta(\phi)$ of $\theta(\pi)$ can be described by using the $L$-parameter of $\pi$ (\cite[Appendix C]{MR3166215}).
\end{description}
By combining these properties with the observation that \textit{every unitary group in odd variables is quasi-split}, we can reduce the problem to the quasi-split case (Theorem \ref{thm:intro}).
More precisely, for a given non-quasi-split unitary group $\G$, we assume that $\G$ is realized by an $\epsilon$-hermitian space $V$ over $F$ (necessarily $V$ is in even variables).
We put $2n$ to be the dimension of this space $V$ and let $V'$ be a $(2n+1)$-dimensional $(-\epsilon)$-hermitian space over $F$.
Then the corresponding unitary group $\U(V')$ is quasi-split.
Thus, for an irreducible smooth representation $\pi$ of $\G=\U(V)$, we can compare its depth with the depth of the $L$-parameter $\phi$ of $\pi$ in the following way:
\begin{enumerate}
\item
By Pan's theorem, the depth of $\pi$ is equal to the depth of its theta lift $\theta(\pi)$.
\item
Since $\U(V')$ is quasi-split, by Theorem \ref{thm:intro}, the depth of $\theta(\pi)$ is equal to the depth of its $L$-parameter $\theta(\phi)$.
\item
By Gan--Ichino's description, we can compare the depth of $\theta(\phi)$ with that of $\phi$.
\end{enumerate}
\[
\xymatrix{
\text{$\pi$: representation of $\U(V)$}\quad\ar@{<~>}[rr]^-{\text{LLC for $\U(V)$}}\ar@{~>}[d]_-{\theta}&&\quad\text{$\phi$: $L$-parameter of $\U(V)$}\ar@{~>}[d]_-{\theta}\\
\text{$\theta(\pi)$: representation of $\U(V')$}\quad\ar@{<~>}[rr]^-{\text{LLC for $\U(V')$}}&&\quad\text{$\theta(\phi)$: $L$-parameter of $\U(V')$}.
}
\]

However, in carrying out this strategy, we have to take care of the difference between the normalization of the local theta correspondence used in Pan's result and that in Gan--Ichino's result.
Recall that the local theta correspondence is obtained by considering the Weil representation of the metaplectic group $\Mp(W)$ for the symplectic space $W:=V\otimes V'$.
More precisely, the metaplectic group $\Mp(W)$ is a covering group of the symplectic group $\Sp(W)$ and contains covering groups $\widetilde{\U(V)}$ of $\U(V)$ and $\widetilde{\U(V')}$ of $\U(V)$.
Then, by restricting the Weil representation of $\Mp(W)$ to $\widetilde{\U(V)}\times\widetilde{\U(V')}$, we get a correspondence between representations of $\widetilde{\U(V)}$ and those of $\widetilde{\U(V')}$.
\[
\xymatrix{
\widetilde{\U(V)}\times\widetilde{\U(V')} \ar[r]\ar[d]& \Mp(W)\ar[d]\\
\U(V)\times\U(V') \ar[r]& \Sp(W).
}
\]
To make this to be a correspondence between representations of $\U(V)$ and those of $\U(V')$, we have to choose splittings of the coverings $\widetilde{\U(V)}\rightarrow\U(V)$ and $\widetilde{\U(V')}\rightarrow\U(V')$.
The important point here is that such splittings are \textit{not} canonical.
Namely, there are several ways to construct them.
The depth preserving result of Pan is based on Pan's splittings constructed in \cite{MR1847153} by using the generalized lattice model of the Weil representation.
On the other hand, the result of Gan--Ichino is based on Kudla's splitting constructed in \cite{MR1286835} by using the Schr\"odinger model of the Weil representation.
Therefore, in order to combine these two results, we have to compute the difference between these two kinds of splittings.

In this paper, we first recall Pan's construction of his splittings for the metaplectic covers of unitary groups (Section \ref{sec:pre}) and compute the difference between Pan's splittings and Kudla's splittings (Section \ref{sec:compare}).
Then, according to the strategy explained as above, we show the depth preserving property of the local Langlands correspondence for non-quasi-split unitary groups (Section \ref{sec:depth}).

\medbreak
\noindent{\bfseries Acknowledgment.}\quad
The author expresses his gratitude to his advisor Yoichi Mieda for his warm encouragement and a lot of helpful comments on a draft version of this paper.
This work was carried out with the support from the Program for Leading Graduate Schools, MEXT, Japan.
This work was also supported by JSPS Research Fellowship for Young Scientists and KAKENHI Grant Number 17J05451.

\setcounter{tocdepth}{2}
\tableofcontents

\medbreak
\noindent{\bfseries Notation.}\quad
Let $p$ be a prime number.
In this paper, we always assume that $p$ is not equal to $2$.
For a $p$-adic field $F$, we denote its ring of integers, maximal ideal, residue field, and valuation by $\mathcal{O}_{F}$, $\mathfrak{p}_{F}$, $\mathbf{f}_{F}$, and $\ord_{F}(-)$, respectively.

We fix a quadratic extension $E/F$ of $p$-adic fields and a uniformizer $\varpi_{F}$ of $F^{\times}$.
We denote the Galois conjugation of $E/F$ by $\tau$, and the norm map $\Nr_{E/F}$ shortly by $\Nr$.
We write $\varepsilon_{E/F}$ for the quadratic character of $F^{\times}$ corresponding to the extension $E/F$.

\section{Metaplectic covers of classical groups and their splittings}\label{sec:pre}

In this section, we recall the notion of metaplectic covers of classical groups, their splittings, and the local theta correspondence.

\subsection{Metaplectic covers of symplectic groups and the Weil representations}\label{subsec:meta}

Let $(W,\langle-,-\rangle)$ be a symplectic space over $F$, that is a finite-dimensional vector space $W$ over $F$ with a nondegenerate symplectic form $\langle-,-\rangle$ on $W$.
Then we have the corresponding symplectic group 
\[
\Sp(W):=\{g\in\GL(W)\mid \langle gw,gw'\rangle=\langle w,w'\rangle \text{ for every $w,w'\in W$}\}
\]
and the \textit{Heisenberg group} $\mathbb{H}(W)=W\times F$ with the multiplication given by
\[
(w_{1},t_{1})\cdot(w_{2},t_{2})
:=
\Bigl(w_{1}+w_{2},t_{1}+t_{2}+\frac{1}{2}\langle w_{1},w_{2}\rangle\Bigr).
\]
Then the center of the group $\mathbb{H}(W)$ is given by $\{0\}\times F$.
By Stone--von Neumann's theorem, for every nontrivial additive character $\psi$ of $F$, we have a unique (up to isomorphism) irreducible smooth representation $(\rho_{\psi}, \mathcal{S}_{\psi})$ of $\mathbb{H}(W)$ with central character $\psi$.

We define the twist $(\rho_{\psi}^{g},\mathcal{S}_{\psi})$ of $(\rho_{\psi},\mathcal{S}_{\psi})$ via $g\in\Sp(W)$ by
\[
\rho_{\psi}^{g}(w,t)(v):=\rho_{\psi}(gw,t)(v) \quad\text{for $(w,t)\in\mathbb{H}(W)$ and $v\in \mathcal{S}_{\psi}$}.
\]
Then we define the \textit{metaplectic cover} of $\Sp(W)$ to be the group
\[
\Mp(W):=\{(g,M)\in \Sp(W)\times\GL(\mathcal{S}_{\psi}) \mid M\colon \rho_{\psi}\cong\rho_{\psi}^{g}\}.
\]
Note that, by Stone--von Neumann's theorem, for every $g\in \Sp(W)$, there always exists $M\in\GL(\mathcal{S}_{\psi})$ satisfying $M\colon\rho_{\psi}\cong\rho_{\psi}^{g}$.
Moreover, by Schur's lemma, such an $M$ is unique up to a scalar multiple.
In other words, we have an exact sequence
\[
1
\rightarrow
\C^{\times}
\rightarrow
\Mp(W)
\rightarrow
\Sp(W)
\rightarrow
1,
\]
where the second map is given by $z\mapsto(1, z\cdot\id_{\mathcal{S}_{\psi}})$ and the third map is given by $(g,M)\mapsto g$, and a Cartesian diagram
\[
\xymatrix{
\Mp(W) \ar[r]^-{\omega_{\psi}}\ar[d]& \GL(\mathcal{S}_{\psi})\ar[d]\\
\Sp(W)\ar[r]& \GL(\mathcal{S}_{\psi})/\C^{\times}.
}
\]
We call the representation $\omega_{\psi}; (g,M)\mapsto M$ the \textit{Weil representation} of $\Mp(W)$.

In the rest of this paper, we fix a nontrivial additive character $\psi$ of $F$.

\subsection{Local theta correspondence for metaplectic covers}\label{subsec:theta}
We consider an $\epsilon$-hermitian space $(V,h)$ over $E$ and an $\epsilon'$-hermitian space $(V',h')$ over $E$, where $\epsilon$ and $\epsilon'$ are elements of $\{\pm1\}$ satisfying $\epsilon\epsilon'=-1$.
We denote the corresponding unitary groups by $\U(V)$ and $\U(V')$.
Namely, we put
\[
\U(V) := \{g\in\GL(V) \mid h(gv_{1},gv_{2})=h(v_{1},v_{2}) \text{ for every $v_{1},v_{2}\in V$}\}
\]
(we define $\U(V')$ in the same way).
We can regard the space $W:=V\otimes_{E}V'$ as a symplectic space over $F$ with the symplectic form
\[
\langle v_{1}\otimes v'_{1}, v_{2}\otimes v'_{2}\rangle
:=
\frac{1}{2} \mathrm{Tr}_{E/F}\bigl(h(v_{1},v_{2})\cdot \tau(h'(v'_{1},v'_{2}))\bigr).
\]
Then the pair $(\U(V),\U(V'))$ is a \textit{reductive dual pair} in the symplectic group $\Sp(W)$ and we have a natural map
\[
\U(V)\times \U(V') \rightarrow \Sp(W).
\]
In particular, we have embeddings $\iota_{V'}\colon\U(V)\rightarrow \Sp(W)$ and $\iota_{V}\colon\U(V')\rightarrow \Sp(W)$.
We define the metaplectic cover $\widetilde{\U(V)}$ to be the pull back of $\iota_{V'}(\U(V))$ via the map $\Mp(W)\rightarrow\Sp(W)$ (we define $\widetilde{\U(V')}$ in a similar way).

Then, by considering the pull back of the Weil representation $\omega_{\psi}$ of $\Mp(W)$ via
\[
\widetilde{\U(V)}\times\widetilde{\U(V')} \rightarrow \widetilde{\U(V)}\cdot\widetilde{\U(V')}\subset\Mp(W),
\]
we get a correspondence between irreducible admissible representations of $\widetilde{\U(V)}$ and those of $\widetilde{\U(V')}$.
More precisely, for every irreducible admissible representation $\widetilde{\pi_{V}}$ of $\widetilde{\U(V)}$, the maximal $\widetilde{\pi_{V}}$-isotypic quotient of $\omega_{\psi}|_{\widetilde{\U(V)}\times\widetilde{\U(V')}}$ is of the form 
\[
\widetilde{\pi_{V}}\boxtimes \Theta_{V,V',\psi}(\widetilde{\pi_{V}}),
\]
where $\Theta_{V,V',\psi}(\widetilde{\pi_{V}})$ is an admissible representation of $\widetilde{\U(V')}$, which is called the \textit{big theta lift} of $\widetilde{\pi_{V}}$ and possibly zero.
If $\Theta_{V,V',\psi}(\widetilde{\pi_{V}})$ is not zero, then $\Theta_{V,V',\psi}(\widetilde{\pi_{V}})$ has a unique irreducible quotient which is denoted by $\theta_{V,V',\psi}(\widetilde{\pi_{V}})$ and called the \textit{small theta lift}.
Moreover, for irreducible admissible representations $\widetilde{\pi_{V,1}}$ and $\widetilde{\pi_{V,2}}$ of $\widetilde{\U(V)}$, $\theta_{V,V',\psi}(\widetilde{\pi_{V,1}})$ is equivalent to $\theta_{V,V',\psi}(\widetilde{\pi_{V,2}})$ if and only if $\widetilde{\pi_{V,1}}$ is equivalent to $\widetilde{\pi_{V,2}}$ (so called the \textit{Howe duality}).
The correspondence between irreducible admissible representations of $\widetilde{\U(V)}$ and those of $\widetilde{\U(V')}$ obtained in this way is called the \textit{local theta correspondence}.
These results are basically based on the works in \cite{MR1041060} and \cite{MR1159105}.
See also, for example, \cite[Section 5]{MR3166215} for a summary of the theory of the local theta correspondence.

\subsection{Splittings of the metaplectic covers}\label{subsec:spl}
Recall that, by Stone--von Neumann's theorem and Schur's lemma, for every $g\in\Sp(W)$, an element $M\in\GL(\mathcal{S}_{\psi})$ satisfying $(g,M)\in\Mp(W)$ exists uniquely up to a scalar multiple.
Now we suppose that we have a special choice of such $M_{g}\in\GL(\mathcal{S}_{\psi})$ for each $g\in\Sp(W)$ satisfying $M_{1}=\id_{\mathcal{S}_{\psi}}$.
Then, as sets, we can identify $\Mp(W)$ with $\Sp(W)\times\C^{\times}$ via
\[
\Sp(W)\times\C^{\times} \cong \Mp(W);\quad (g,z)\mapsto (g, zM_{g}).
\]
Moreover, if we define a $2$-cocycle $c(-,-)\colon \Sp(W)\times\Sp(W)\rightarrow\C^{\times}$ by
\[
c(g_{1},g_{2})
:=
M_{g_{1}} M_{g_{2}}M_{g_{1}g_{2}}^{-1} \in \C^{\times},
\]
then the group structure of $\Mp(W)$ can be described in $\Sp(W)\times\C^{\times}$ as
\[
(g_{1},z_{1})\cdot(g_{2},z_{2})
=
\bigl(g_{1}g_{2}, z_{1}z_{2}c(g_{1},g_{2})\bigr).
\]

We say that a function $\beta_{V'}\colon \U(V)\rightarrow\C^{\times}$ is a \textit{splitting} of the cocycle $c$ if we have
\[
c(g,g')
=
\beta_{V'}(gg') \beta_{V'}(g)^{-1} \beta_{V'}(g')^{-1}
\]
for every $g,g'\in\Sp(W)$.
Note that this condition is equivalent to the condition that the map
\[
\widetilde{\beta}_{V'}\colon \U(V) \rightarrow \widetilde{\U(V)};\quad
g\mapsto \bigl(\iota_{V'}(g), \beta_{V'}(g)M_{g}\bigr)
\]
is a group homomorphism (i.e., a section of the covering $\widetilde{\U(V)}\rightarrow\U(V)$).
Furthermore, if we have such an splitting, then we get a group isomorphism
\[
\U(V)\times\C^{\times}\cong\widetilde{\U(V)};\quad
(g,z)\mapsto \bigl(\iota_{V'}(g), z\beta_{V'}(g)M_{g}\bigr).
\]
We say that a splitting $\beta_{V'}$ is \textit{admissible} if the pullback of every admissible representation of $\widetilde{\U(V)}$ via $\widetilde{\beta}_{V'}$ is an admissible representation of $\U(V)$.

Now we suppose that we have such admissible splittings $\beta_{V'}$ for $\U(V)$ and $\beta_{V}$ for $\U(V')$.
Then we can regard the local theta correspondence as a correspondence between representations of $\U(V)$ and those of $\U(V')$ in the following way.
First, we take an irreducible admissible representation $\pi_{V}$ of $\U(V)$.
We write $W_{\pi_{V}}$ for a vector space where $\pi_{V}$ is realized.
Since we have an isomorphism $\widetilde{\U(V)}\cong\U(V)\times\C^{\times}$ obtained from the splitting $\beta_{V'}$, we can extend $\pi_{V}$ to an irreducible admissible representation $\widetilde{\pi_{V}}$ of $\widetilde{\U(V)}$ on $W_{\pi_{V}}$ so that 
\[
\widetilde{\pi_{V}}(g,z):=\pi_{V}(g)\circ z\cdot\id_{W_{\pi_{V}}}
\]
for $(g,z)\in \U(V)\times\C^{\times}\cong\widetilde{\U(V)}$.
By considering the local theta correspondence, we get the small theta lift $\theta_{V,V',\psi}(\widetilde{\pi_{V}})$ of $\widetilde{\pi_{V}}$, which is an admissible representation of $\widetilde{\U(V')}$.
Now we assume that $\theta_{V,V',\psi}(\widetilde{\pi_{V}})$ is not zero.
Then, by noting that the central part $z\in\C^{\times}$ of $\Mp(W)$ acts on $\mathcal{S}_{\psi}$ via $z\cdot\id_{\mathcal{S}_{\psi}}$, the pullback of $\theta_{V,V',\psi}(\widetilde{\pi_{V}})$ to $\U(V')$ via the fixed splitting $\beta_{V}$ is irreducible.
Thus we get an irreducible admissible representation of $\U(V')$.

We remark that there is another way to make the local theta correspondence to be a correspondence between representations of $\U(V)$ and those of $\U(V')$.
That is, we first consider the homomorphism
\[
\U(V)\times\U(V')
\hookrightarrow
\widetilde{\U(V)}\times\widetilde{\U(V')} \rightarrow \widetilde{\U(V)}\cdot\widetilde{\U(V')}\subset\Mp(W)
\]
obtained from the fixed splittings $\beta_{V'}$ and $\beta_{V}$.
Then we define the local theta correspondence between $\U(V)$ and $\U(V')$ in the same manner as in Section \ref{subsec:theta} by using the pullback of the Weil representation $\omega_{\psi}$ of $\Mp(W)$ to $\U(V)\times\U(V')$.
Later (in Section \ref{subsec:key}), we combine two key results on the local theta correspondence which have been established by Gan--Ichino (\cite{MR3166215}) and Pan (\cite{MR1909608}).
In \cite{MR1909608}, the local theta correspondence for $\U(V)$ and $\U(V')$ constructed in the first way is used.
On the other hand, in \cite{MR3166215}, the correspondence constructed in the second way is used.
However, since the central part $z\in\C^{\times}$ of $\Mp(W)$ acts on $\mathcal{S}_{\psi}$ via $z\cdot\id_{\mathcal{S}_{\psi}}$, these two constructions coincide.


\subsection{Schr\"odinger models vs.\ generalized lattice models}\label{subsec:vs}
As explained in the previous subsection, by taking admissible splittings, we can get the local theta correspondence between $\U(V)$ and $\U(V')$.
This correspondence depends on the admissible splittings and there are several ways to construct such splittings explicitly.
In this paper, we use the following two kinds of splittings:

\begin{description}
\item[Kudla's splitting for the Schr\"odinger model]
The first one is Kudla's splitting constructed by using the Schr\"odinger model of the representation $\rho_{\psi}$.
We take a complete polarization $W=X\oplus Y$ (i.e., $X$ and $Y$ are totally isotropic subspaces).
Then we can realize the representation $\rho_{\psi}$ of $\mathbb{H}(W)$ on the space $\mathcal{S}(Y)$ of Schwartz functions (i.e., locally constant and compactly supported functions) on $Y$.
Moreover, by using this realization, we can construct an isomorphism $\rho_{\psi}\cong\rho_{\psi}^{g}$ for each $g\in\Sp(W)$ explicitly.
We write $M_{g}^{Y}$ for this isomorphism.
Here we do not recall how to realize the representation $\rho_{\psi}$ on the space $\mathcal{S}(Y)$ and how to construct $M_{g}^{Y}$.
See, for example, \cite[Section 2.2]{MR1847153}.
Then, as explained in Section \ref{subsec:spl}, we get a $2$-cocycle with respect to $\{M_{g}^{Y}\}_{g\in\Sp(W)}$ (called \textit{Ranga Rao's $2$-cocycle}).
For this $2$-cocycle, Kudla constructed an admissible splitting $\beta_{V'}^{Y}\colon\U(V)\rightarrow\C^{\times}$ explicitly in \cite{MR1286835}.
Here note that, to define the splitting $\beta_{V'}^{Y}$, we have to take a character $\chi_{V'}$ on $E^{\times}$ satisfying $\chi_{V'}|_{F^{\times}}=\varepsilon_{E/F}^{\dim{V'}}$, and the splitting $\beta_{V'}^{Y}$ depends on the choice of such a character.
Also note that the resulting splitting $\widetilde{\beta}_{V'}$ does not depend on the choice of a polarization $W=X\oplus Y$ (see, for example, \cite[Proposition A.1]{MR1327161}).

\item[Pan's splitting for the generalized lattice model]
The second one is Pan's splitting constructed by using the generalized lattice model of the representation $\rho_{\psi}$.
By taking good lattices $L$ of $V$ and $L'$ of $V'$, we can define a good lattice $B$ of $W$ (see \cite[Section 1.3]{MR1847153} for the definition of good lattices).
Then we can realize the representation $\rho_{\psi}$ of $\mathbb{H}(W)$ on the space $\mathcal{S}(B)$ of Schwartz functions on $B$.
Moreover, by using this realization, we can construct an isomorphism $M_{g}^{B}\colon\rho_{\psi}\cong\rho_{\psi}^{g}$ for each $g\in\Sp(W)$ explicitly (see \cite[Section 2.3]{MR1847153}).
Therefore we get a $2$-cocycle with respect to $\{M_{g}^{B}\}_{g\in\Sp(W)}$.
For this $2$-cocycle, Pan constructed an admissible splitting $\beta_{V'}^{L}\colon\U(V)\rightarrow\C^{\times}$ explicitly in \cite{MR1847153}.
\end{description}

In this paper, we fix a character $\bm{\chi}$ of $E^{\times}$ satisfying $\bm{\chi}|_{F\times}=\varepsilon_{E/F}$ and always take $\chi_{V'}$ (in the definition of Kudla's splitting) to be $\bm{\chi}^{\dim{V'}}$.

Let us consider the difference between the local theta correspondence with respect to Kudla's splittings and that with respect to Pan's splittings.
We first take an isomorphism $\Psi\colon\mathcal{S}(Y)\cong\mathcal{S}(B)$ as irreducible representations of $\mathbb{H}(W)$ (note that this is unique up to a scalar multiple by the irreducibility of $\rho_{\psi}$).
By using $\Psi$, we define a function $\alpha_{V'}\colon\U(V)\rightarrow\C^{\times}$ to be the function satisfying
\[
\alpha_{V'}(g)\cdot \Psi\circ M_{\iota_{V'}(g)}^{B}= M_{\iota_{V'}(g)}^{Y}\circ\Psi
\]
for every $g\in\U(V)$.
Then, since $\beta_{V'}^{Y}$ is an admissible splitting for the Schr\"odinger model, the function $\alpha_{V'}\beta_{V'}^{Y}$ is an admissible splitting for the generalized lattice model.
We define the function $\xi_{V'}\colon\U(V)\rightarrow\C^{\times}$ to be the ratio of this splitting $\alpha_{V'}\beta_{V'}^{Y}$ to Pan's splitting $\beta_{V'}^{L}$:
\[
\xi_{V'}:= \alpha_{V'}\beta_{V'}^{Y}(\beta_{V'}^{L})^{-1}.
\]
Then this is a character of $\U(V)$ and the relationship between the local theta correspondences with respect to Kudla's splitting and Pan's splitting is described as follows:

\begin{prop}\label{prop:compare}
Let $\pi_{V}^{L}$ and $\pi_{V'}^{L'}$ are irreducible admissible representations of $\U(V)$ and $\U(V')$, respectively.
We assume that $\pi_{V}^{L}$ and $\pi_{V'}^{L'}$ are related under the local theta correspondence with respect to Pan's splittings $\beta_{V'}^{L}$ and $\beta_{V}^{L'}$.
Then $\pi_{V}^{L}\otimes\xi_{V'}$ and $\pi_{V'}^{L'}\otimes\xi_{V}$ are related under the local theta correspondence with respect to Kudla's splittings $\beta_{V'}^{Y}$ and $\beta_{V}^{Y}$.
\end{prop}

\begin{proof}
We assume that $\pi_{V}^{L}$ is the restriction of an irreducible admissible representation $\widetilde{\pi_{V}}$ of $\widetilde{\U(V)}$ to $\U(V)$ via Pan's splitting $\beta_{V'}^{L}$.
Namely, for every $g\in\U(V)$, we have
\[
\pi_{V}^{L}(g) = \widetilde{\pi_{V}}\bigl(\iota_{V'}(g),\beta_{V'}^{L}(g)M_{\iota_{V'}(g)}^{B}\bigr).
\]
Since we have $\xi_{V'}=\alpha_{V'}\beta_{V'}^{Y}(\beta_{V'}^{L})^{-1}$ and $\alpha_{V'}(g)\cdot M_{\iota_{V'}(g)}^{B}= M_{\iota_{V'}(g)}^{Y}$ (here we omit $\Psi$ from the notation), we have 
\begin{align*}
\pi_{V}^{L}\otimes\xi_{V'}(g)
&=\widetilde{\pi_{V}}\bigl(\iota_{V'}(g),\beta_{V'}^{L}(g)M_{\iota_{V'}(g)}^{B}\xi_{V'}(g)\bigr)\\
&=\widetilde{\pi_{V}}\bigl(\iota_{V'}(g),\beta_{V'}^{Y}(g)M_{\iota_{V'}(g)}^{Y}\bigr).
\end{align*}
In other words, $\pi_{V}^{L}\otimes\xi_{V'}$ is the restriction of $\widetilde{\pi_{V}}$ to $\U(V)$ via Kudla's splitting $\beta_{V'}^{Y}$.
On the other hand, if we assume that $\pi_{V'}^{L'}$ is the restriction of an irreducible admissible representation $\widetilde{\pi_{V'}}$ of $\widetilde{\U(V')}$ to $\U(V')$ via Pan's splitting $\beta_{V}^{L'}$, then $\pi_{V'}^{L'}\otimes\xi_{V}$ is the restriction of $\widetilde{\pi_{V'}}$ to $\U(V')$ via Kudla's splitting $\beta_{V}^{Y}$ by the same argument.
This completes the proof.
\end{proof}


\section{Comparison of Kudla's splitting and Pan's splitting}\label{sec:compare}

Our aim in this section is to compute the ratio $\xi_{V'}$ of Kudla's splitting to Pan's splitting explicitly.

\subsection{Maximal open compact subgroup $\U(V)_{L}$ and its character $\zeta_{V'}$}\label{subsec:parah}
In this subsection, we recall the definitions and basic properties of the maximal open compact subgroup $\U(V)_{L}$ of $\U(V)$ and its character $\zeta_{V'}$, which play important roles in Pan's construction of the splitting $\beta_{V'}^{L}$ for the generalized lattice models of the Weil representations.

For the fixed good lattice $L$ of $V$, we denote the corresponding maximal open compact subgroup of $\U(V)$ by $\U(V)_{L}$:
\[
\U(V)_{L}:=\{g\in\U(V) \mid g\cdot L=L\}.
\]
We denote the image of $\U(V)_{L}$ under the determinant map by $E_{L}$.
When $E$ is unramified over $F$, this group $E_{L}$ is equal to 
\[
E^{1}:=\{t\in E^{\times}\mid \Nr(t)=1\}.
\]
When $E$ is ramified over $F$, the group $E_{L}$ is a subgroup of $E^{1}$ of index two and given by
\[
(E^{1})^{+}
:=\{t\in E^{1}\mid t\equiv1\bmod\mathfrak{p}_{E}\}.
\]
We set $\SU(V)_{L}:=\U(V)_{L}\cap\SL(V)$.


Now we define the character $\zeta_{V'}$ on $\U(V)_{L}$ as follows.
First, we consider the quotient of $\U(V)_{L}$ by its commutator subgroup $[\U(V)_{L},\U(V)_{L}]$.
Then we have the following exact sequence of abelian groups:
\[
1
\rightarrow
\SU(V)_{L}/[\U(V)_{L},\U(V)_{L}]
\rightarrow
\U(V)_{L}/[\U(V)_{L},\U(V)_{L}]
\rightarrow
\U(V)_{L}/\SU(V)_{L}
\rightarrow
1.
\]
If we fix a section of the third homomorphism, then we get an isomorphism
\[
\U(V)_{L}/[\U(V)_{L},\U(V)_{L}]
\cong
\bigl(\SU(V)_{L}/[\U(V)_{L},\U(V)_{L}]\bigr)
\times
\bigl(\U(V)_{L}/\SU(V)_{L}\bigr).
\tag{$\ast$}
\]
We first define a character $\zeta_{V'}$ on $\SU(V)_{L}/[\U(V)_{L},\U(V)_{L}]$ of order $2$ as in the manner of \cite[Section 3.3]{MR1847153}.
Then, by using the above isomorphism $(\ast)$, we extend $\zeta_{V'}$ to $\U(V)_{L}/[\U(V)_{L},\U(V)_{L}]$ trivially.
Finally, by inflating it to $\U(V)_{L}$, we define a character $\zeta_{V'}$ on $\U(V)_{L}$.

Here we note that the definition of $\zeta_{V'}$ depends on the choice of the isomorphism $(\ast)$.
Later (in Section \ref{subsec:sect}), we explain our choice of this isomorphism.
We also note that the restriction of $\zeta_{V'}$ to the pro-$p$ part does not depend on the choice of the isomorphism $(\ast)$.
To be more precise, we identify $\U(V)_{L}/\SU(V)_{L}$ with $E_{L}$ via $\det$:
\[
\det\colon\U(V)_{L}/\SU(V)_{L}\xrightarrow{\cong}E_{L},
\]
and put
\[
E_{L}^{+}:=\{t\in E_{L} \mid t\equiv1\bmod\mathfrak{p}_{E}\}
\]
(note that we have $E_{L}\supsetneq E_{L}^{+}$ only when $E$ is unramified over $F$).
Let $\U(V)_{L}^{+}$ be the inverse image of $E_{L}^{+}$ via the determinant map $\U(V)_{L}\twoheadrightarrow E_{L}$:
\[
\xymatrix{
1 \ar[r]&\SU(V)_{L}/[\U(V)_{L},\U(V)_{L}]\ar[r]&\U(V)_{L}/[\U(V)_{L},\U(V)_{L}]\ar[r]&E_{L}\ar[r]&1\\
1 \ar[r]&\SU(V)_{L}/[\U(V)_{L},\U(V)_{L}]\ar[r]\ar[u]^-{=}&\U(V)_{L}^{+}/[\U(V)_{L},\U(V)_{L}]\ar[r]\ar@{^{(}->}[u]&E_{L}^{+}\ar@{^{(}->}[u]\ar[r]&1.
}
\]
Then, since $E_{L}^{+}$ is a pro-$p$ group and $\SU(V)_{L}/[\U(V)_{L},\U(V)_{L}]$ is a prime-to-$p$ group (see \cite[Section 1.4]{MR1847153} for the details), the exact sequence of abelian groups
\[
1
\rightarrow
\SU(V)_{L}/[\U(V)_{L},\U(V)_{L}]
\rightarrow
\U(V)_{L}^{+}/[\U(V)_{L},\U(V)_{L}]
\rightarrow
E_{L}^{+}
\rightarrow
1
\]
splits canonically.
Thus the extension of $\zeta_{V'}$ from $\SU(V)_{L}/[U(V)_{L}, \U(V)_{L}]$ to $\U(V)_{L}^{+}/[\U(V)_{L},\U(V)_{L}]$ is determined canonically.

\subsection{Reduction to a computation of $\alpha_{V'}\beta_{V'}^{Y}$}\label{subsec:ab}

To compute the difference $\xi_{V'}$ between Kudla's splitting and Pan's splitting, we next recall the definition of Pan's splitting $\beta_{V'}^{Y}$.
To define $\beta_{V'}^{Y}$, we use the character $\zeta_{V'}\colon\U(V)_{L}\rightarrow\C^{\times}$ constructed in the previous subsection.

First, by the definition of the generalized lattice models, $\alpha_{V'}\beta_{V'}^{Y}$ is a character on $\U(V)_{L}$.
On the other hand, by Proposition 3.3 in \cite{MR1847153}, we have
\[
\alpha_{V'}\beta_{V'}^{Y}|_{\SU(V)_{L}}
\equiv
\zeta_{V'}|_{\SU(V)_{L}}.
\]
In other words, the map $\alpha_{V'}\beta_{V'}^{Y}\zeta_{V'}^{-1}$ 
is a character on $\U(V)_{L}$ factoring through the determinant map.
We write $\xi'_{V'}$ for the induced character of $E_{L}=\det(\U(V)_{L})$:
\[
\xymatrix{
\U(V)_{L} \ar[rr]^-{\alpha_{V'}\beta_{V'}^{Y}\zeta_{V'}^{-1}}\ar[rd]_-{\det}&& \C^{\times}\\
&E_{L}\ar[ru]_-{\xi'_{V'}}&
}
\]
We take an extension of $\xi'_{V'}$ to $E^{1}$ and denote it again by $\xi'_{V'}$.
Then Pan's splitting $\beta_{V'}^{L}\colon\U(V)\rightarrow\C^{\times}$ is defined by
\[
\beta_{V'}^{L}
:=
(\xi'_{V'}\circ\det)^{-1}\alpha_{V'}\beta_{V'}^{Y}.
\]
Therefore, by this definition, the ratio $\xi_{V'}$ of Kudla's splitting to Pan's splitting is given by
\[
\xi_{V'}
:=\alpha_{V'}\beta_{V'}^{Y}\cdot(\beta_{V'}^{L})^{-1}
=\xi'_{V'}\circ\det.
\]

Thus our task is to determine $\xi'_{V'}$ explicitly.
However, the extension of $\xi'_{V'}$ from $E_{L}$ to $E^{1}$ was taken arbitrary.
Furthermore, the character $\zeta_{V'}$ on $\U(V)_{L}$ is defined by using the noncanonical isomorphism $(\ast)$ (see Section \ref{subsec:parah}).
Therefore $\xi'_{V'}$ is defined canonically only on $E_{L}^{+}$.
Thus our essential task is to determine $\xi'_{V'}|_{E_{L}^{+}}$.
\[
\xymatrix{
\U(V)_{L}^{+}\ar@{^{(}->}[r]\ar[d]_-{\det}&\U(V)_{L} \ar[rr]^-{\alpha_{V'}\beta_{V'}^{Y}\zeta_{V'}^{-1}}\ar[d]_-{\det}&& \C^{\times}\\
E_{L}^{+}\ar@{^{(}->}[r]&E_{L}\ar[rru]_-{\xi'_{V'}}&&
}
\]
To do this, it suffices to compute $\alpha_{V'}\beta_{V'}^{Y}\zeta_{V'}^{-1}$ at the image of a section of the determinant map $\U(V)_{L}^{+}\rightarrow E_{L}^{+}$.
However, by the definition of $\zeta_{V'}$ and an explanation in the final paragraph in Section \ref{subsec:parah}, $\zeta_{V'}$ is necessarily trivial on the image of such a section.
Thus it is enough to compute $\alpha_{V'}\beta_{V'}^{Y}$ on the image of a section of the determinant map $\U(V)_{L}^{+}\rightarrow E_{L}^{+}$.

\subsection{Section of the determinant map}\label{subsec:sect}
In this subsection, we construct an explicit section of the determinant map $\U(V)_{L}^{+}\rightarrow E_{L}^{+}$.

When $V$ is $1$-dimensional, the unitary group $\U(V)$ is isomorphic to $E^{1}$ canonically.
Thus there is nothing to do.

We next consider the case where $V$ is $2$-dimensional isotropic.
In this case, we consider a matrix representation of the unitary group $\U(V)$ with respect to a basis $\{v_{1},v_{2}\}$ of $V$ satisfying
\begin{itemize}
\item
$h(v_{1},v_{1})=h(v_{2},v_{2})=0$, and
\item
$h(v_{1},v_{2})=1$,
\end{itemize}
where $h$ is the $\epsilon$-hermitian form of the $\epsilon$-hermitian space $V$.
Furthermore, we assume that the fixed good lattice $L$ is of the form
\[
\mathfrak{p}_{E}^{\nu_{1}}v_{1}\oplus\mathfrak{p}_{E}^{\nu_{2}}v_{2}
\]
for some integers $\nu_{1},\nu_{2}\in\Z$.
Here note that it is enough to treat only such a case since every good lattice is of the form $g\cdot L$ for some element $g\in\U(V)$ and a good lattice $L\subset V$ of the above form (see Lemma 1.3 (ii) in \cite{MR1847153}).
To construct a splitting in this case, we start from recalling the structure of the quotient of $1+\mathfrak{p}_{E}$ by $1+\mathfrak{p}_{F}$.
First, since $E$ is quadratic over $F$ and $p$ is not equal to $2$, we may assume that $E=F(\delta)$ and $\mathcal{O}_{E}=\mathcal{O}_{F}[\delta]$, where $\delta$ is a square root of a non-square element $\Delta\in F^{\times}\setminus F^{\times2}$ satisfying 
\[
\ord_{F}(\Delta)=
\begin{cases}
0 & \text{if $E$ is unramified over $F$},\\
1 & \text{if $E$ is ramified over $F$}.
\end{cases}
\]
Then we can define a section of the canonical surjection 
\[
(1+\mathfrak{p}_{E})
\twoheadrightarrow
(1+\mathfrak{p}_{E})/(1+\mathfrak{p}_{F})
\]
as follows:
\begin{description}
\item[The case where $E/F$ is unramified]
In this case, we have
\begin{itemize}
\item
$\mathfrak{p}_{E}=\varpi_{F}(\mathcal{O}_{F}+\delta\mathcal{O}_{F})$, and
\item
$\mathfrak{p}_{F}=\varpi_{F}\mathcal{O}_{F}$.
\end{itemize}
Thus the subgroup  
\[
\{1+\varpi_{F}\delta x \mid x\in \mathcal{O}_{F}\}
\]
of $1+\mathfrak{p}_{E}$ is a set of representatives of $(1+\mathfrak{p}_{E})/(1+\mathfrak{p}_{F})$ and gives a section of the surjection $1+\mathfrak{p}_{E}\twoheadrightarrow(1+\mathfrak{p}_{E})/(1+\mathfrak{p}_{F})$.

\item[The case where $E/F$ is ramified]
In this case, we have
\begin{itemize}
\item
$\mathfrak{p}_{E}=\delta\mathcal{O}_{F}+\delta^{2}\mathcal{O}_{F}$, and
\item
$\mathfrak{p}_{F}=\delta^{2}\mathcal{O}_{F}$.
\end{itemize}
Thus the subgroup
\[
\{1+\delta x \mid x\in \mathcal{O}_{F}\}
\]
of $1+\mathfrak{p}_{E}$ is a set of representatives of $(1+\mathfrak{p}_{E})/(1+\mathfrak{p}_{F})$ and gives a section of the surjection $1+\mathfrak{p}_{E}\twoheadrightarrow(1+\mathfrak{p}_{E})/(1+\mathfrak{p}_{F})$.
\end{description}
Now we construct a section of the determinant map.
By Hilbert's theorem 90, we have an isomorphism
\[
E^{\times}/F^{\times}\xrightarrow{\cong} E^{1};\quad k\mapsto k/\tau(k).
\]
Moreover, the subset $E_{L}^{+}$ of the right-hand side corresponds to the subset $(1+\mathfrak{p}_{E})/(1+\mathfrak{p}_{F})$ of the left-hand side.
Therefore, by using the above representatives of $(1+\mathfrak{p}_{E})/(1+\mathfrak{p}_{F})$, we can define a section
\[
E_{L}^{+}\rightarrow 1+\mathfrak{p}_{E};\quad t\mapsto k_{t} 
\]
of the map $1+\mathfrak{p}_{E}\twoheadrightarrow(1+\mathfrak{p}_{E})/(1+\mathfrak{p}_{F})\cong E_{L}^{+}$.
Namely, $k_{t}$ is the unique element of 
\[
\begin{cases}
\{1+\varpi_{F}\delta x \mid x\in \mathcal{O}_{F}\}
& \text{if $E/F$ is unramified},\\
\{1+\delta x \mid x\in \mathcal{O}_{F}\}& \text{if $E/F$ is ramified}
\end{cases}
\]
satisfying $k_{t}/\tau(k_{t})=t$.
By using this, we define a map from $E_{L}^{+}$ to $\U(V)_{L}^{+}$ by
\[
E_{L}^{+} \rightarrow \U(V)_{L}^{+};\quad t\mapsto
\begin{pmatrix}
k_{t}&0\\
0&\tau(k_{t})^{-1}
\end{pmatrix}.
\]
Then this is a section of the determinant map $\det\colon\U(V)^{+}_{L}\rightarrow E_{L}^{+}$.

Finally, we consider the general case.
In this case, by Lemma 1.3 (iii) in \cite{MR1847153}, $V$ has an $L$-admissible decomposition into an orthogonal sum of $2$-dimensional isotropic $\epsilon$-hermitian spaces and $1$-dimensional anisotropic $\epsilon$-hermitian subspaces (see \cite[Section 1.3]{MR1847153} for the definition of the $L$-admissibility).
We take such a decomposition of $V$ and a decomposition of the good lattice $L$ into good lattices $L_{i}$ of $V_{i}$ determined by the decomposition of $V$:
\[
V=\bigoplus_{i\in I}V_{i}
\supset
L=\bigoplus_{i\in I}L_{i}.
\]
We fix one index $i_{0}\in I$.
Then we can regard the group $\U(V_{i_{0}})_{L_{i_{0}}}$ as a subgroup of $\U(V)_{L}$.
Thus, by using the sections for $1$ or $2$-dimensional cases constructed above, we get a section of the determinant map from $\U(V)_{L}^{+}$ to $E_{L}^{+}$:
\[
E_{L}\rightarrow
\U(V_{i_{0}})_{L_{i_{0}}}^{+}\hookrightarrow
\U(V)_{L}^{+}.
\]


\subsection{Computation of the ratio of two splittings}\label{subsec:comp}

Now we compute the difference between Kudla's splitting and Pan's splitting.
Recall that, in Section \ref{subsec:vs}, we defined the character $\xi_{V'}$ of $\U(V)$ to be the ratio of Kudla's splitting to Pan's splitting.
Then the character $\xi_{V'}$ factors through the determinant character $\det\colon\U(V)\rightarrow E^{1}$ and can be written as $\xi_{V'}=\xi'_{V'}\circ\det$.

On the other hand, we fixed a character $\chi_{V'}$ of $E^{\times}$ to define Kudla's splitting.
By the isomorphism of Hilbert's theorem 90, the chain $E_{L}^{+}\subset E_{L}\subset E^{1}$ can be identified with a chain of $E^{\times}/F^{\times}$ as follows:
\[
\xymatrix{
(1+\mathfrak{p}_{E})/(1+\mathfrak{p}_{F})\ar@{^{(}->}[r]\ar[d]^-{\cong}&\mathcal{O}_{E}^{\times}/\mathcal{O}_{F}^{\times} \ar@{^{(}->}[r]\ar[d]^-{\cong}& E^{\times}/F^{\times}\ar[d]^-{\cong}\\
E_{L}^{+}\ar@{^{(}->}[r]&E_{L}\ar@{^{(}->}[r]&E^{1}.
}
\]
Since $\chi_{V'}$ is a character on $E^{\times}$ whose restriction to $F^{\times}$ is equal to either $\mathbbm{1}$ or $\varepsilon_{E/F}$, the restriction of $\chi_{V'}$ to $1+\mathfrak{p}_{E}$ factors through the quotient $(1+\mathfrak{p}_{E})/(1+\mathfrak{p}_{F})$.
Therefore, we can regard $\chi_{V'}|_{1+\mathfrak{p}_{E}}$ as a character of $E_{L}^{+}$.
We denote it by $\chi^{+}_{V'}$.


\begin{prop}\label{prop:depth0}
We have
\[
\xi'_{V'}|_{E_{L}^{+}}=\chi^{+}_{V'}.
\]
\end{prop}

In the rest of this section, we prove this proposition.

We consider a decomposition $V=\bigoplus_{i}V_{i}$ of $V$ into $1$-dimensional subspaces or $2$-dimensional isotropic subspaces as in the previous subsection.
We denote the section of the determinant map from $\U(V)_{L}^{+}$ to $E_{L}^{+}$ constructed in the previous subsection by $\mu$:
\[
\mu\colon 
E_{L}^{+}\rightarrow
\U(V_{i_{0}})_{L_{i_{0}}}^{+}\hookrightarrow
\U(V)_{L}^{+}.
\]
Then, in order to show Proposition \ref{prop:depth0}, it is enough to check that
\[
\xi_{V'}\bigl(\mu(t)\bigr)
=
\chi^{+}_{V'}(t)
\]
for every $t\in E_{L}^{+}$.

We have $\xi_{V'}=\alpha_{V'}\beta_{V'}^{Y}\zeta_{V'}^{-1}$ and $\zeta_{V'}$ is trivial on $\U(V)_{L}^{+}$ as explained in Section \ref{subsec:ab}.
Thus our task is to show
\[
\alpha_{V'}\beta_{V'}^{Y}\bigl(\mu(t)\bigr)
=
\chi^{+}_{V'}(t).
\]

If we take a maximal isotropic subspace $Y$ of $V\otimes_{E}V'$ to be the direct sum of maximal isotropic subspaces $Y_{i}$ of $V_{i}\otimes_{E}V'$, then, for every element $g\in\U(V)$ of the form
\[
g=(g_{i})_{i} \in \prod_{i}\U(V_{i})\subset\U(V),
\]
we have
\[
\alpha_{V'}\beta_{V'}^{Y}(g)
=
\prod_{i}\alpha_{V'}\beta_{V'}^{Y_{i}}(g_{i})
\]
(see \cite[Section 3.6]{MR1847153} for the details).
In particular, by the definition of $\mu$, we have
\[
\alpha_{V'}\beta_{V'}^{Y}\bigl(\mu(t)\bigr)
=
\alpha_{V'}\beta_{V'}^{Y_{i_{0}}}\bigl(\mu(t)\bigr).
\]
Namely, it is enough to consider the case where $V$ is either $1$-dimensional or $2$-dimensional isotropic.

Furthermore, we consider a decomposition of $V'\bigoplus_{j}V'_{j}$ into two-dimensional isotropic or $1$-dimensional subspaces, and take a maximal isotropic subspace $Y$ of $V\otimes_{E}V'$ to be the direct sum of maximal isotropic subspaces $Y_{j}$ of $V\otimes_{E}V'_{j}$.
Then, for every element $g\in\U(V)$, we have 
\[
\alpha_{V'}\beta_{V'}^{Y}(g)
=
\prod_{j}\alpha_{V'_{j}}\beta_{V'_{j}}^{Y_{j}}(g_{j}).
\]

Therefore, by our choice of the characters $\chi_{V'}$ and $\chi_{V_{j'}}$ for each $V'_{j}$ (see Section \ref{subsec:vs}), it is enough to show that 
\[
\alpha_{V'_{j}}\beta_{V'_{j}}^{Y_{j}}\bigl(\mu(t)\bigr)
=
\chi_{V'_{j}}^{+}(t)
\]
for each $j$.

In summary, to prove Proposition \ref{prop:compare} for general $V$ and $V'$, it is enough to show it for the case where $V$ and $V'$ are $1$-dimensional or $2$-dimensional isotropic.
We check it by a case-by-case computation.

\subsubsection{The case where $E/F$ is unramified}
We first consider the case where $E$ is an unramified extension of $F$.

\begin{lem}[The case where $\dim{V}=1$ and $\dim{V'}=1$]\label{lem:ur11}
Let $\chi_{V'}$ be a character of $E^{\times}$ satisfying $\chi_{V'}|_{F^{\times}}=\varepsilon_{E/F}$.
Let $t$ be an element of $E_{L}^{+}$ and $k_{t}$ the element constructed in Section $\ref{subsec:sect}$ satisfying $k_{t}/\tau(k_{t})=t$.
Then we have
\[
\alpha_{V'}\beta_{V'}^{Y}\bigl(\mu(t)\bigr)
=
\chi^{+}_{V'}(t)
=
\chi_{V'}(k_{t}).
\]
\end{lem}

\begin{proof}
We write $k_{t}=1+\delta b$, where $b$ is an element of $\mathfrak{p}_{F}$.
Here, since it is enough to consider the case where $t\neq1$, we may assume that $b\neq0$.
Similarly, we write $t=x+\delta y$ by using elements $x,y\in\mathcal{O}_{F}$.
Note that we have $x\equiv1\bmod\mathfrak{p}_{F}$ and $y$ is given by $\frac{2b}{\Nr(k_{t})}$ (in particular, $y\neq0$).

We put $\sigma:=\ord(\psi_{0})+\ord_{F}(\rho)$.
Here $\rho\in F^{\times}$ is an element determined by $V$ and $V'$ (see \cite[Section 5.2]{MR1847153}) and $\ord(\psi_{0})$ is the level of the additive character $\psi_{0}$ of $F$ defined by $\psi_{0}(x):=\psi(\frac{x}{2})$.
Then, by Lemma 6.1 in \cite{MR1847153}, we have
\[
\beta_{V'}^{Y}\bigl(\mu(t)\bigr)
=
\begin{cases}
\chi_{V'}((t-1)\delta)& \text{if $\sigma$ is even,}\\
-\chi_{V'}((t-1)\delta)& \text{if $\sigma$ is odd, and}
\end{cases}
\]
\[
\alpha_{V'}\bigl(\mu(t)\bigr)
=
\gamma_{F}(-\Delta,\psi_{0})\cdot(\Delta,y\rho)_{F}\cdot\omega_{0}(-1)^{\ord(\psi_{0})},
\]
where $\gamma_{F}(-,-)$ is the Weil constant (see, for example, Section 1.5 in \cite{MR1847153}), $(-,-)_{F}$ is the Hilbert symbol, and $\omega_{0}$ is the nontrivial quadratic character of $\mathbf{f}_{F}^{\times}$.
Here we remark that, in Lemma 6.1 in \cite{MR1847153} the formula of $\alpha_{V'}$ is incorrect, and that $2$ in the Hilbert symbol is not necessary.
This mistake arises from the definitions of $I_{1}$ and $I_{2}$, which are used in the proof of \cite[Lemma 5.9]{MR1847153}.
We rewrite this result in terms of $k_{t}=1+\delta b$.

First we consider $\beta_{V'}^{Y}(\mu(t))$.
As we assume that $\chi_{V'}|_{F^{\times}}=\varepsilon_{E/F}$, $\chi_{V'}$ is trivial on $\Nr(E^{\times})$.
In particular, we have $\chi_{V'}(k_{t})=\chi_{V'}(\tau(k_{t}))^{-1}$.
Thus, since we have
\[
(t-1)\delta
=\frac{k_{t}-\tau(k_{t})}{\tau(k_{t})}\cdot\delta
=\frac{2\Delta b}{\tau(k_{t})},
\]
we get
\[
\beta_{V'}^{Y}\bigl(\mu(t)\bigr)
=
(-1)^{\sigma} \chi_{V'}(2\Delta b) \chi_{V'}(k_{t}).
\]

Next we consider $\alpha_{V'}(\mu(t))$.
By a formula for the Weil constant (see, e.g., Section 1.5 (6) in \cite{MR1847153}), we have
\begin{align*}
\gamma(-\Delta,\psi_{0})
&=
\bigl(\omega_{0}(-\overline{\Delta})\gamma_{\mathbf{f}_{F}}(\overline{\psi_{0}})\bigr)^{\ord_{F}(-\Delta)}
\omega_{0}(-\overline{\Delta})^{\ord(\psi_{0})}
\omega_{0}(-1)^{\ord_{F}(-\Delta)\ord(\psi_{0})}\\
&=
(-1)^{\ord(\psi_{0})}\omega_{0}(-1)^{\ord(\psi_{0})}.
\end{align*}
Here we put $\overline{\Delta}\in\mathbf{f}_{F}^{\times}$ to be the reduction of $\Delta\varpi_{F}^{-\ord_{F}(\Delta)}$.
By noting that $(\Delta,y\rho)_{F}=\chi_{V'}(y\rho)$ and that $y=\frac{2b}{\Nr(k_{t})}$, we get
\begin{align*}
\gamma_{F}(-\Delta,\psi_{0})(\Delta,y\rho)_{F}\omega_{0}(-1)^{\ord(\psi_{0})}
&=(-1)^{\ord(\psi_{0})}\chi_{V'}\bigl(2b\rho\Nr(k_{t})^{-1}\bigr)\\
&=(-1)^{\ord(\psi_{0})}\chi_{V'}(2b\rho).
\end{align*}
Thus we have
\begin{align*}
\alpha_{V'}\beta_{V'}^{Y}\bigl(\mu(t)\bigr)
&=
(-1)^{\ord(\psi_{0})}\chi_{V'}(2b\rho)
(-1)^{\sigma} \chi_{V'}(2\Delta b) \chi_{V'}(k_{t})\\
&=(-1)^{\ord(\psi_{0})+\sigma}\chi_{V'}(\rho\Delta)\chi_{V'}(k_{t}).
\end{align*}
Finally, by noting that $\chi_{V'}(\rho\Delta)=(-1)^{\ord_{F}(\rho)}$, we get
\[
(-1)^{\ord(\psi_{0})+\sigma}\chi_{V'}(\rho\Delta)\chi_{V'}(k_{t})
=
\chi_{V'}(k_{t}).
\]

\end{proof}

\begin{lem}[The case where $\dim{V}=1$ and $\dim{V'}=2$]\label{lem:ur12}
Let $\chi_{V'}$ be a character of $E^{\times}$ satisfying $\chi_{V'}|_{F^{\times}}=\mathbbm{1}$.
Let $t$ be an element of $E_{L}^{+}$ and $k_{t}$ the element constructed in Section $\ref{subsec:sect}$ satisfying $k_{t}/\tau(k_{t})=t$.
Then we have
\[
\alpha_{V'}\beta_{V'}^{Y}\bigl(\mu(t)\bigr)
=
\chi^{+}_{V'}(t)
=
\chi_{V'}(k_{t}).
\]
\end{lem}

\begin{proof}
By Lemma 6.2 in \cite{MR1847153}, we have
\[
\beta_{V'}^{Y}\bigl(\mu(t)\bigr)
=
\chi_{V'}\bigl((t-1)\delta\bigr)
\quad\text{and}\quad
\alpha_{V'}\bigl(\mu(t)\bigr)
=
1.
\]
Thus, by the same computation as in the previous case, we get 
\[
\alpha_{V'}\beta_{V'}^{Y}(g)
=
\chi_{V'}(2\Delta bk_{t}).
\]
Since we assume $\chi_{V'}|_{F^{\times}}\equiv\mathbbm{1}$, we get the assertion.
\end{proof}

\begin{lem}[The case where $\dim{V}=2$ and $\dim{V'}=1$]\label{lem:ur21}
Let $\chi_{V'}$ be a character of $E^{\times}$ satisfying $\chi_{V'}|_{F^{\times}}=\varepsilon_{E/F}$.
Let $t$ be an element of $E_{L}^{+}$ and $k_{t}$ the element constructed in Section $\ref{subsec:sect}$ satisfying $k_{t}/\tau(k_{t})=t$.
Then we have
\[
\alpha_{V'}\beta_{V'}^{Y}\bigl(\mu(t)\bigr)
=
\chi^{+}_{V'}(t)
=
\chi_{V'}(k_{t}).
\]
\end{lem}

\begin{proof}
This is nothing other than Lemma 6.4 in \cite{MR1847153}.
\end{proof}

\begin{lem}[The case where $\dim{V}=2$ and $\dim{V'}=2$]\label{lem:ur22}
Let $\chi_{V'}$ be a character of $E^{\times}$ satisfying $\chi_{V'}|_{F^{\times}}=\mathbbm{1}$.
Let $t$ be an element of $E_{L}^{+}$ and $k_{t}$ the element constructed in Section $\ref{subsec:sect}$ satisfying $k_{t}/\tau(k_{t})=t$.
Then we have
\[
\alpha_{V'}\beta_{V'}^{Y}\bigl(\mu(t)\bigr)
=
\chi^{+}_{V'}(t)
=
\chi_{V'}(k_{t}).
\]
\end{lem}

\begin{proof}
This is nothing other than Lemma 6.5 in \cite{MR1847153}.
\end{proof}

By the arguments in the beginning of this subsection and Lemmas \ref{lem:ur11}, \ref{lem:ur12}, \ref{lem:ur21}, and \ref{lem:ur22}, the proof of Proposition \ref{prop:depth0} in the case where $E$ is unramified over $F$ is completed.

\subsubsection{The case where $E/F$ is ramified}
We next consider the case where $E$ is a ramified extension of $F$.

\begin{lem}[The case where $\dim{V}=1$ and $\dim{V'}=1$]\label{lem:ram11}
Let $\chi_{V'}$ be a character of $E^{\times}$ satisfying $\chi_{V'}|_{F^{\times}}=\varepsilon_{E/F}$.
Let $t$ be an element of $E_{L}^{+}$ and $k_{t}$ the element constructed in Section $\ref{subsec:sect}$ satisfying $k_{t}/\tau(k_{t})=t$.
Then we have
\[
\alpha_{V'}\beta_{V'}^{Y}\bigl(\mu(t)\bigr)
=
\chi^{+}_{V'}(t)
=
\chi_{V'}(k_{t}).
\]
\end{lem}

\begin{proof}
We write $k_{t}=1+\delta b$, where $b$ is an element of $\mathcal{O}_{F}$.
Here, since it is enough to consider the case where $t\neq1$, we may assume that $b\neq0$.
Similarly, we write $t=x+\delta y$ by using elements $x,y\in\mathcal{O}_{F}$.
Note that we have $x\equiv1\bmod\mathfrak{p}_{F}$ and $y$ is given by $\frac{2b}{\Nr(k_{t})}$ (in particular, $y\neq0$).

By Lemma 7.1 in \cite{MR1847153}, we have
\[
\beta_{V'}^{Y}\bigl(\mu(t)\bigr)
=
\chi_{V'}\bigl((t-1)\delta\bigr)\cdot(\Delta,\rho)_{F}\cdot\gamma_{F}(\Delta,\psi_{0})^{-1}, \text{ and}
\]
\[
\alpha_{V'}\bigl(\mu(t)\bigr)
=
\gamma_{F}(-\Delta,\psi_{0})\cdot(\Delta,y\rho)_{F}\cdot\omega_{0}(-1)^{\ord(\psi_{0})}.
\]
We rewrite this result in terms of $k_{t}=1+\delta b$.

First we consider $\beta_{V'}^{Y}(\mu(t))$.
By the same computation as in Lemma \ref{lem:ur11}, we get
\[
\chi_{V'}\bigl((t-1)\delta\bigr)
=
\chi_{V'}(2\Delta b) \chi_{V'}(k_{t}).
\]
We have $(\Delta,\rho)_{F}=\chi_{V'}(\rho)$ and, by a formula for the Weil constant (see \cite[Section 1.5 (6)]{{MR1847153}}), 
\begin{align*}
\gamma_{F}(\Delta,\psi_{0})^{-1}
&=
\bigl(\omega_{0}(\overline{\Delta})\gamma_{\mathbf{f}_{F}}(\overline{\psi_{0}})\bigr)^{-\ord_{F}(\Delta)}
\omega_{0}(\overline{\Delta})^{-\ord(\psi_{0})}
\omega_{0}(-1)^{-\ord_{F}(\Delta)\ord(\psi_{0})}\\
&=
\gamma_{\mathbf{f}_{F}}(\overline{\psi_{0}})^{-1}
\omega_{0}(-1)^{-\ord(\psi_{0})}
\omega_{0}(\overline{\Delta})^{-1-\ord(\psi_{0})}.
\end{align*}
Thus we get
\[
\beta_{V'}^{Y}\bigl(\mu(t)\bigr)
=
\chi_{V'}(2\Delta b\rho) \chi_{V'}(k_{t})\gamma_{\mathbf{f}_{F}}(\overline{\psi_{0}})^{-1}
\omega_{0}(-1)^{-\ord(\psi_{0})}
\omega_{0}(\overline{\Delta})^{-1-\ord(\psi_{0})}.
\]

Next we consider $\alpha_{V'}(\mu(t))$.
Again by a formula for the Weil constant, we have
\begin{align*}
\gamma_{F}(-\Delta,\psi_{0})
&=
\bigl(\omega_{0}(-\overline{\Delta})\gamma_{\mathbf{f}_{F}}(\overline{\psi_{0}})\bigr)^{\ord_{F}(-\Delta)}
\omega_{0}(-\overline{\Delta})^{\ord(\psi_{0})}
\omega_{0}(-1)^{\ord_{F}(-\Delta)\ord(\psi_{0})}\\
&=
\omega_{0}(-1)^{\ord(\psi_{0})}\gamma_{\mathbf{f}_{F}}(\overline{\psi_{0}})
\omega_{0}(-\overline{\Delta})^{1+\ord(\psi_{0})}.
\end{align*}
Thus, by noting that $(\Delta,y\rho)_{F}=\chi_{V'}(y\rho)$ and that $y=\frac{2b}{\Nr(k_{t})}$, we get
\[
\alpha_{V'}\bigl(\mu(t)\bigr)
=
\gamma_{\mathbf{f}_{F}}(\overline{\psi_{0}})\chi_{V'}(2b\rho)
\omega_{0}(-\overline{\Delta})^{1+\ord(\psi_{0})}.
\]

Therefore we have
\[
\alpha_{V'}\beta_{V'}^{Y}\bigl(\mu(t)\bigr)
=
\omega_{0}(-1)\chi_{V'}(\Delta) \chi_{V'}(k_{t}).
\]
By noting that $\chi_{V'}|_{F^{\times}}=\varepsilon_{E/F}$ and that $E$ is ramified over $F$, we have $\chi_{V'}(\Delta)=\omega_{0}(-1)$.
Thus we get
\[
\alpha_{V'}\beta_{V'}^{Y}\bigl(\mu(t)\bigr)
=
\chi_{V'}(k_{t}).
\]
\end{proof}

\begin{lem}[The case where $\dim{V}=1$ and $\dim{V'}=2$]\label{lem:ram12}
Let $\chi_{V'}$ be a character of $E^{\times}$ satisfying $\chi_{V'}|_{F^{\times}}=\mathbbm{1}$.
Let $t$ be an element of $E_{L}^{+}$ and $k_{t}$ the element constructed in Section $\ref{subsec:sect}$ satisfying $k_{t}/\tau(k_{t})=t$.
Then we have
\[
\alpha_{V'}\beta_{V'}^{Y}\bigl(\mu(t)\bigr)
=
\chi^{+}_{V'}(t)
=
\chi_{V'}(k_{t}).
\]
\end{lem}

\begin{proof}
The assertion follows from Lemma 7.2 in \cite{MR1847153} and the same argument as in the proof of Lemma \ref{lem:ur12}.
\end{proof}

\begin{lem}[The case where $\dim{V}=2$ and $\dim{V'}=1$]\label{lem:ram21}
Let $\chi_{V'}$ be a character of $E^{\times}$ satisfying $\chi_{V'}|_{F^{\times}}=\varepsilon_{E/F}$.
Let $t$ be an element of $E_{L}^{+}$ and $k_{t}$ the element constructed in Section $\ref{subsec:sect}$ satisfying $k_{t}/\tau(k_{t})=t$.
Then we have
\[
\alpha_{V'}\beta_{V'}^{Y}\bigl(\mu(t)\bigr)
=
\chi^{+}_{V'}(t)
=
\chi_{V'}(k_{t}).
\]
\end{lem}

\begin{proof}
This is nothing other than Lemma 7.4 in \cite{MR1847153}.
\end{proof}

\begin{lem}[The case where $\dim{V}=2$ and $\dim{V'}=2$]\label{lem:ram22}
Let $\chi_{V'}$ be a character of $E^{\times}$ satisfying $\chi_{V'}|_{F^{\times}}=\mathbbm{1}$.
Let $t$ be an element of $E_{L}^{+}$ and $k_{t}$ the element constructed in Section $\ref{subsec:sect}$ satisfying $k_{t}/\tau(k_{t})=t$.
Then we have
\[
\alpha_{V'}\beta_{V'}^{Y}\bigl(\mu(t)\bigr)
=
\chi^{+}_{V'}(t)
=
\chi_{V'}(k_{t}).
\]
\end{lem}

\begin{proof}
This is nothing other than Lemma 7.5 in \cite{MR1847153}.
\end{proof}

By the arguments in the beginning of this subsection and Lemmas \ref{lem:ram11}, \ref{lem:ram12}, \ref{lem:ram21}, and \ref{lem:ram22}, the proof of Proposition \ref{prop:depth0} in the case where $E$ is ramified over $F$ is completed.

\section{Depth preserving property of the local Langlands correspondence for non-quasi-split unitary groups}\label{sec:depth}
In this section, we apply our result on a comparison of splittings of metaplectic covers in the previous section to the depth preserving problem of the local Langlands correspondence for non-quasi-split unitary groups.

\subsection{Inner forms and pure inner forms of unitary groups}\label{subsec:inner}
We first recall a classification of inner forms and pure inner forms of unitary groups briefly.
See, for example, \cite[Section 2]{MR3202556}, \cite[Section 2]{MR3166215}, or \cite[Section 0.3.3]{KMSW} for the details.

Let $\G$ be the quasi-split unitary group with respect to a quadratic extension $E/F$ in $N$ variables.
Then the isomorphism classes of inner forms of $\G$ are classified by the set
\[
H^{1}(F,\G_{\ad}),
\]
and the order of this set is given by
\[
\begin{cases}
1 & \text{if $N$ is odd},\\
2 & \text{if $N$ is even}.
\end{cases}
\]
In other words, when $N$ is odd, then every unitary group with respect to $E/F$ in $N$ variables is quasi-split and isomorphic to $\G$.
When $N$ is even, there are two isomorphism classes of unitary groups with respect to $E/F$ in $N$ variables; one is quasi-split and isomorphic to $\G$, and the other is non-quasi-split.

On the other hand, the set of pure inner forms of $\G$ is given by
\[
H^{1}(F,\G).
\]
The order of this set is always equal to $2$, and we can identify $H^{1}(F,\G)$ with the set of equivalence classes of $\epsilon$-hermitian spaces in $N$ variables, for any $\epsilon$.

The natural map
\[
H^{1}(F,\G)\rightarrow H^{1}(F,\G_{\ad})
\]
is surjective, and the pure inner form corresponding to an $\epsilon$-hermitian space $V$ maps to the inner form $\U(V)$ of $\G$.
For an element $\alpha$ of $H^{1}(F, \G)$, we denote the corresponding inner form of $\G$ by $\G_{\alpha}$.
We put $e(\G_{\alpha})$ to be the Kottwitz sign of $\G_{\alpha}$, which is an invariant defined in \cite{MR697075}.
Note that, in our situation, this is given by
\[
e(\G_{\alpha})
=
\begin{cases}
1 & \text{if $\G_{\alpha}$ is quasi-split},\\
-1 & \text{if $\G_{\alpha}$ is non-quasi-split}.
\end{cases}
\] 

\begin{table}[htb]
  \begin{center}
    \begin{tabular}{|c|c|c|c|} \hline
	 size $N$ & pure inner form & inner form & Kottwitz sign \\ \hline
	 even & $1$ & quasi-split & $1$\\ \cline{2-4}
	  & $\alpha\neq1$ & non-quasi-split & $-1$\\ \cline{1-4}
 	 odd & $1$ or $\beta\neq1$ & quasi-split & $1$\\ \cline{1-4}
    \end{tabular}
  \end{center}
\end{table}

\subsection{Local Langlands correspondence for unitary groups}\label{subsec:LLC}
We next recall the local Langlands correspondence for unitary groups according to the formalism of ``Vogan $L$-packets'' based on the pure inner form.

Let $\G$ be the quasi-split unitary group with respect to $E/F$ in $N$ variables.
We write $\Phi(\G)$ for the set of $\widehat{\G}$-conjugacy classes of $L$-parameters of $\G$.
We put
\[
\Pi(\G):=\bigsqcup_{\alpha\in H^{1}(F,\G)}\Pi(\G_{\alpha}),
\]
where $\Pi(\G_{\alpha})$ is the set of equivalence classes of irreducible smooth representations of $\G_{\alpha}(F)$.

Then the local Langlands correspondence for $\G$, which was established by Mok (\cite{MR3338302}) in the quasi-split case, and announced by Kaletha--Minguez--Shin--White (\cite{KMSW}) in the non-quasi-split case (see also, for example, \cite[Section 12.3]{BP}) is formulated as follows:
\begin{thm}[the local Langlands correspondence for $\G$]\label{thm:LLC}
We have a natural partition of $\Pi(\G)$ into finite sets
\[
\Pi(\G)
=
\bigsqcup_{\phi\in\Phi(\G)}\Pi_{\phi}
\]
characterized by the following properties:
\begin{enumerate}
\item
For every tempered $L$-parameter $\phi$ and every pure inner form $\alpha\in H^{1}(F,\G)$, the finite set $\Pi_{\phi}^{\alpha}:=\Pi_{\phi}\cap\Pi(\G_{\alpha})$ is stable.
Namely, the sum $\Theta_{\phi}^{\alpha}$ of the characters of irreducible smooth representations of $\G_{\alpha}(F)$ belonging to $\Pi_{\phi}^{\alpha}$ is a stable distribution.
\item
For every tempered $L$-parameter $\phi$, the stable distribution $\Theta_{\phi}^{1}$ on $\G_{1}(F)$ satisfies the endoscopic character relation with the twisted character of the endoscopic lift of $\Pi_{\phi}$ to a general linear group.
To be more precise, by regarding $\phi$ as an $L$-parameter of $\mathrm{Res}_{E/F}\GL_{N}$ by the standard base change embedding from the $L$-group of $\G$ to that of $\mathrm{Res}_{E/F}\GL_{N}$, we get the corresponding conjugate self-dual representation $\pi_{\phi}^{\GL_{N}}$ of $\GL_{N}(E)$.
We denote its twisted character with respect to a fixed Whittaker data of $\mathrm{Res}_{E/F}\GL_{N}$ by $\Theta_{\phi,\theta}^{\GL_{N}}$.
Then, for every strongly regular semisimple elements $g\in \GL_{N}(E)$ and $h\in \G_{1}(F)$ such that $h$ is a norm of $g$, we have
\[
\Theta_{\phi,\theta}^{\GL_{N}}(g)
=
\Theta_{\phi}^{1}(h).
\]

\item
For every tempered $L$-parameter $\phi$ and every pure inner form $\alpha\in H^{1}(F,\G)$, the stable distribution $\Theta_{\phi}^{\alpha}$ on $\G_{\alpha}(F)$ is the transfer of the stable distribution $e(\G_{\alpha})\Theta_{\phi}^{1}$ on $\G_{1}(F)$.
More precisely, for every strongly regular semisimple elements $g_{\alpha}\in\G_{\alpha}(F)$ and $g_{1}\in\G_{1}(F)$ which are stably conjugate, we have
\[
\Theta_{\phi}^{\alpha}(g_{\alpha})
=
e(\G_{\alpha})\Theta_{\phi}^{1}(g_{1}).
\]
\item
The partition is compatible with the theory of the Langlands quotient.
\end{enumerate}
\end{thm}

We call each finite set $\Pi_{\phi}$ the $L$-packet for $\phi$.
For an irreducible representation $\pi\in\Pi(\G)$, we often write $\phi_{\pi}$ for its $L$-parameter (that is, the unique $L$-parameter satisfying $\pi\in\Pi_{\phi_{\pi}}$).

\begin{rem}
In \cite{MR3338302}, the endoscopic character relation (Theorem \ref{thm:LLC} (2)) is stated in terms of distribution characters.
By using Weyl's integration formula, we can rewrite it as an equality of characters (i.e., functions on the strongly regular semisimple loci of $\G_{1}(F)$ and $\GL_{N}(E)$) (see, for example, Section 5 in \cite{MR3067291}).
Then the obtained equality of the characters is more complicated than the equality in Theorem \ref{thm:LLC} (2).
However, by noting the following points, we can simplify it and get the equality of the form in Theorem \ref{thm:LLC} (2) (see, for example, \cite[Corollary 4.8]{Oi:2016aa}):
\begin{itemize}
\item
For every strongly regular semisimple (in the sense of twisted conjugacy) element $g\in\GL_{N}(E)$, there exists at most one stable conjugacy classes of norms of $g$ in $\G_{1}(F)$.
\item
For every strongly regular semisimple elements $g\in \GL_{N}(E)$ and $h\in \G_{1}(F)$ such that $h$ is a norm of $g$, their Weyl discriminants coincide.
\item
Since we consider the endoscopic lifting with respect to the standard base change embedding, the Kottwitz--Shelstad transfer factor is trivial.
\end{itemize}
\end{rem}

The following observation is trivial, but plays an important role in the proof of our main theorem.
\begin{lem}\label{lem:odd}
We assume that $N$ is odd.
For the nontrivial pure inner form $\beta\in H^{1}(F,\G)$, we identify $\Pi(\G_{1})$ with $\Pi(\G_{\beta})$ by the isomorphism $\G_{1}\cong\G_{\beta}$ $($determined uniquely up to an inner automorphism$)$.
Then, for every $L$-parameter $\phi$ of $\G$, we have
\[
\Pi_{\phi}^{1}=\Pi_{\phi}^{\beta}.
\]
\end{lem}

\begin{proof}
Since the local Langlands correspondence for non-tempered $L$-parameters is constructed from that for tempered $L$-parameters via the theory of the Langlands quotient (Theorem \ref{thm:LLC} (4)), it suffices to show the assertion in the tempered case.
However, by noting that $e(\G_{\beta})=1$, the assertion immediately follows from the stability of $L$-packets (Theorem \ref{thm:LLC} (1)), the transfer relation (Theorem \ref{thm:LLC} (3)), and the linear independence of the characters of irreducible smooth representations.
\end{proof}

\subsection{Local Langlands correspondence and character twists}\label{subsec:twist}
Before we consider the depth preserving problem of the local Langlands correspondence for unitary groups, we investigate the relationship between the local Langlands correspondence and character twists.

Let $\G$ be the quasi-split unitary group with respect to a quadratic extension $E/F$ in $N$ variables.
Then we can identify the set of $L$-parameters of $\G$ with the set of ``conjugate self-dual $L$-parameters of $\GL_{N}(E)$ with parity $(-1)^{N-1}$'' by composing them with the standard base change $L$-embedding from the $L$-group of $\G$ to that of $\mathrm{Res}_{E/F}\GL_{N}$ (see, for example, \cite[Lemma 2.2.1]{MR3338302} for the details).
On the other hand, if a conjugate self-dual character $\eta$ of $E^{\times}$ satisfies $\eta|_{F^{\times}}\equiv\mathbbm{1}$, then the twist via $\eta$ does not change the parity of a conjugate self-dual $L$-parameter of $\GL_{N}(E)$.
Namely, for any conjugate self-dual $L$-parameter $\phi$ of $\GL_{N}(E)$ with parity $(-1)^{N-1}$, the parity of $\phi\otimes\eta$ is given by $(-1)^{N-1}$ (here we regard $\eta$ as a character of the Weil group $W_{E}$ of $E$ by the local class field theory).
In particular, the twist $\phi\otimes\eta$ again belongs to the image of the standard base change lifting.
Hence we can regard $\phi\otimes\eta$ as an $L$-parameter of $\G$.
Then the relationship between $L$-packets for $\phi$ and $\phi\otimes\eta$ can be described as follows:

\begin{prop}\label{prop:twist}
Let $\phi$ be a tempered $L$-parameter of $\G$, and $\alpha\in H^{1}(F,\G)$ a pure inner form of $\G$.
Then,  for every character $\eta_{E^{\times}}$ of $E^{\times}$ satisfying $\eta_{E^{\times}}|_{F^{\times}}\equiv\mathbbm{1}$, we have
\[
\Pi_{\phi\otimes\eta_{E^{\times}}}^{\alpha}
=
\Pi_{\phi}^{\alpha}\otimes(\eta_{E^{1}}\circ\det).
\]
Here $\eta_{E^{1}}$ is the character of $E^{1}$ which corresponds to $\eta_{E^{\times}}$ via the isomorphism of Hilbert's theorem 90. 
\end{prop}

\begin{proof}
Before we start to prove the assertion, we note that, for every representation $\pi\in\Pi(\G)$, the character of $\pi\otimes(\eta_{E^{1}}\circ\det)$ is equal to $\Theta_{\pi}\cdot(\eta_{E^{1}}\circ\det)$, where $\Theta_{\pi}$ is the character of $\pi$.
In particular, the sum of the characters belonging to the set in the right-hand side of the equality in the assertion is given by
\[
\Theta_{\phi}^{\alpha}\cdot(\eta_{E^{1}}\circ\det).
\]

We first consider the case where $\alpha$ is the trivial pure inner form.
By the linear independence of the characters of irreducible smooth representations, the endoscopic character relation (Theorem \ref{thm:LLC} (2)) characterizes the $L$-packets for $\G_{1}$.
Therefore it is enough to show that, for every strongly regular semisimple elements $g\in\GL_{N}(E)$ and $h\in\G_{1}(F)$ such that $h$ is a norm of $g$, we have
\[
\Theta_{\phi\otimes\eta_{E^{\times}},\theta}^{\GL_{N}}(g)
=
\Theta_{\phi}^{1}(h)\cdot(\eta_{E^{1}}\circ\det)(h).
\tag{$\ast$}
\]
By noting that
\begin{itemize}
\item 
$\Theta_{\phi\otimes\eta_{E^{\times}},\theta}^{\GL_{N}}$ is equal to $\Theta_{\phi,\theta}^{\GL_{N}}\cdot(\eta_{E^{\times}}\circ\det)$ (this follows from a property of the local Langlands correspondence for $\GL_{N}$ on a character twist), 
\item
$\Theta_{\phi,\theta}^{\GL_{N}}$ and $\Theta_{\phi}^{1}$ satisfy the endoscopic character relation, 
\item
$\det(h)=\det(g)/\tau(\det(g))$ (this follows from the definition of the norm correspondence, see, e.g., \cite[Section 4.2]{Oi:2016aa}), and
\item
$\eta_{E^{1}}(\det(g)/\tau(\det(g)))=\eta_{E^{\times}}(\det(g))$ (this follows from the definition of $\eta_{E^{1}}$),
\end{itemize}
we have
\begin{align*}
\text{LHS of $(\ast)$}
&=
\Theta_{\phi,\theta}^{\GL_{N}}(g)\cdot(\eta_{E^{\times}}\circ\det)(g)\\
&=
\Theta_{\phi}^{1}(h)\cdot(\eta_{E^{1}}\circ\det)(h)
=
\text{RHS of $(\ast)$}.
\end{align*}

We next consider the case where $\alpha$ is the nontrivial pure inner form.
By the linear independence of the characters of irreducible smooth representations, the transfer relation (Theorem \ref{thm:LLC} (3)) characterizes the $L$-packets for $\G_{\alpha}$.
Therefore it is enough to show that, for every strongly regular semisimple elements $g_{\alpha}\in\G_{\alpha}(F)$ and $g_{1}\in\G_{1}(F)$ which are stably conjugate, we have
\[
\Theta_{\phi}^{\alpha}(g_{\alpha})\cdot(\eta_{E^{1}}\circ\det)(g_{\alpha})
=
e(\G_{\alpha})\Theta_{\phi\otimes\eta_{E^{\times}}}^{1}(g_{1}).
\tag{$\ast\ast$}
\]
By noting that
\begin{itemize}
\item 
$\Theta_{\phi}^{\alpha}$ is a transfer of $e(\G_{\alpha})\Theta_{\phi}^{1}$,
\item 
$\Theta_{\phi\otimes\eta_{E^{\times}}}^{1}$ is equal to $\Theta_{\phi}^{1}\cdot(\eta_{E^{1}}\circ\det)$ (the assertion for $\alpha=1$), and
\item
$\det(g_{\alpha})=\det(g_{1})$,
\end{itemize}
we have
\[
\text{LHS of $(\ast\ast)$}
=
e(\G_{\alpha})\Theta_{\phi}^{1}(g_{1})\cdot(\eta_{E^{1}}\circ\det)(g_{1})
=
\text{RHS of $(\ast\ast)$}.
\]
This completes the proof.
\end{proof}

\subsection{Depth preserving property: quasi-split case}\label{subsec:q-spl}
In this subsection, we recall the notion of the depth of representations and a result of \cite{Oi:2018a}

For a connected reductive group $\J$ over $F$, we write $\mathcal{B}(\J,F)$ for the Bruhat--Tits building of $\J$.
For a point $x\in\mathcal{B}(\J,F)$ and $r\in\R_{\geq0}$, we denote the corresponding $r$-th Moy--Prasad filtration of a parahoric subgroup of $\J(F)$ by $J_{x,r}$.
Then, for an irreducible smooth representation $\pi$ of $\J(F)$, its \textit{depth} is defined by
\[
\depth(\pi)
:=
\inf \bigr\{r\in\R_{\geq0} \,\big\vert\, \pi^{J_{x, r+}}\neq0 \text{ for some } x\in\mathcal{B}(\J,F)\bigr\}\in\R_{\geq0},
\]
where $r+$ is $r+\varepsilon$ for a sufficiently small positive number $\varepsilon$.
On the other hand, for an $L$-parameter $\phi$ of $\J$, we define its \textit{depth} to be 
\[
\depth(\phi)
:=
\inf \bigr\{r\in\R_{\geq0} \,\big\vert\, \text{$\phi|_{I_{F}^{r+}}$ is trivial}\bigr\}\in\R_{\geq0}.
\]
Here $I_{F}^{\bullet}$ is the ramification filtration of the inertia subgroup $I_{F}$ of $W_{F}$.

Then, in \cite{Oi:2018a}, by extending a method of Ganapathy--Varma in \cite{MR3709003}, we showed the following:
\begin{thm}[{\cite[Theorem 5.6]{Oi:2018a}}]\label{thm:depth-q-spl}
Let $\G$ be the quasi-split unitary group over $F$ in $N$ variables.
Assume that the residual characteristic $p$ of $F$ is greater than $N+1$.
Then the local Langlands correspondence for $\G_{1}$ preserves the depth of representations.
More precisely, for every $L$-parameter $\phi$ of $\G$ and every $\pi\in\Pi_{\phi}^{1}$, we have
\[
\depth(\pi)=\depth(\phi).
\]
\end{thm}

In the case of odd unitary groups, by using Lemma \ref{lem:odd}, we can immediately generalize this result to Vogan $L$-packets:
\begin{cor}\label{cor:odd}
Let $\G'$ be the quasi-split unitary group over $F$ in $2n+1$ variables, where $n$ is a non-negative integer.
Assume that the residual characteristic $p$ of $F$ is greater than $2n+2$.
Then, for every $L$-parameter $\phi'$ of $\G'$ and every $\pi'\in\Pi_{\phi'}$, we have
\[
\depth(\pi')=\depth(\phi').
\]
\end{cor}

The aim of the rest of this paper is to show the depth preserving property of the local Langlands correspondence for non-quasi-split unitary groups.
Namely, our purpose is to extend the above results to nontrivial pure inner forms of quasi-split unitary groups in even variables.

\subsection{Two key ingredients from the local theta correspondence}\label{subsec:key}

To extend Theorem \ref{thm:depth-q-spl} to pure inner forms of unitary groups, we utilize two important results on the local theta correspondence.

Let $\G$ and $\G'$ be the quasi-split unitary groups with respect to $E/F$ in $N$ and $N+1$ variables, respectively.
Then, for every pure inner forms $\alpha\in H^{1}(F,\G)$ and $\beta\in H^{1}(F,\G')$, we can consider the local theta correspondence between $\widetilde{\G_{\alpha}(F)}$ and $\widetilde{\G'_{\beta}(F)}$.
Here, the groups $\widetilde{\G_{\alpha}(F)}$ and $\widetilde{\G'_{\beta}(F)}$ are the metaplectic covers of $\G_{\alpha}(F)$ and $\G'_{\beta}(F)$, respectively (see Section \ref{subsec:theta}).
If we want to make this to be a correspondence between representations of $\G_{\alpha}(F)$ and $\G'_{\beta}(F)$, as we explained in Section \ref{subsec:spl}, we have to choose admissible splittings $\G_{\alpha}(F)\rightarrow\widetilde{\G_{\alpha}(F)}$ and $\G'_{\beta}(F)\rightarrow\widetilde{\G'_{\beta}(F)}$.
By taking characters $\chi$ and $\chi'$ of $E^{\times}$ satisfying $\chi|_{F^{\times}}=\varepsilon_{E/F}^{N}$ and $\chi'|_{F^{\times}}=\varepsilon_{E/F}^{N+1}$ as in the manner of Section \ref{subsec:vs}, we get Kudla's splittings.
For an irreducible smooth representation $\pi$ of $\G_{\alpha}(F)$, we denote its small theta lift to $\G'_{\beta}(F)$ obtained by using Kudla's splitting by $\theta_{\alpha,\beta,\psi}^{\chi,\chi'}(\pi)$ (recall that the theta correspondence depends also on the choice of a nontrivial additive character $\psi$ of $F$).

The first key ingredient is a description of the local theta correspondence via the local Langlands correspondence, which was established by Gan and Ichino:

\begin{thm}[{\cite[Theorem C.5]{MR3166215}}]\label{thm:GI}
Let $\phi$ be an $L$-parameter of $\G$.
We put $\theta(\phi)$ to be the $L$-parameter of $\G'$ defined by
\[
\theta(\phi):=(\phi\otimes{\chi'}^{-1}\chi)\oplus\chi.
\]
Then we have the following:
\begin{enumerate}
\item
We assume that $\phi$ does not contain $\chi$.
Then, for every $\alpha\in H^{1}(F,\G)$, $\beta\in H^{1}(F,\G')$, and $\pi\in\Pi_{\phi}^{\alpha}$, the small theta lift $\theta_{\alpha,\beta,\psi}^{\chi,\chi'}(\pi)$ of $\pi$ to $\G'_{\beta}(F)$ is not zero.
Moreover, this representation $\theta_{\alpha,\beta,\psi}^{\chi,\chi'}(\pi)$ is irreducible and belongs to $\Pi_{\theta(\phi)}^{\beta}$.
Furthermore, this correspondence gives a bijection from $\Pi_{\phi}$ to $\Pi_{\theta(\phi)}^{\beta}$:
\[
\theta_{-,\beta,\psi}^{\chi,\chi'}\colon \Pi_{\phi}\xrightarrow{1:1}\Pi_{\theta(\phi)}^{\beta}.
\]
\item
We assume that $\phi$ contains $\chi$.
Then, for every $\alpha\in H^{1}(F,\G)$ and $\pi\in\Pi_{\phi}^{\alpha}$, there exists a unique $\beta\in H^{1}(F,\G')$ such that the small theta lift $\theta_{\alpha,\beta,\psi}^{\chi,\chi'}(\pi)$ of $\pi$ to $\G'_{\beta}(F)$ is not zero.
Moreover, this representation $\theta_{\alpha,\beta,\psi}^{\chi,\chi'}(\pi)$ is irreducible and belongs to $\Pi_{\theta(\phi)}^{\beta}$.
Furthermore, this correspondence gives a bijection from $\Pi_{\phi}$ to $\Pi_{\theta(\phi)}$:
\[
\theta_{-,-,\psi}^{\chi,\chi'}\colon \Pi_{\phi}\xrightarrow{1:1}\Pi_{\theta(\phi)}.
\]
\end{enumerate}
\end{thm}

Recall that, by choosing good lattices of hermitian spaces, we get Pan's splittings and the local theta correspondence with respect to them.
For an irreducible smooth representation $\pi$ of $\G_{\alpha}(F)$, we denote its small theta lift to $\G'_{\beta}(F)$ obtained by using Pan's splittings by $\theta_{\alpha,\beta,\psi}(\pi)$.

The second key ingredient is the depth preserving property of the local theta correspondence, which was proved by Pan:

\begin{thm}[{\cite[Theorem 12.2]{MR1909608}}]\label{thm:Pan}
Let $\alpha\in H^{1}(F,\G)$ and $\beta\in H^{1}(F,\G')$.
For $\pi\in\Pi(\G_{\alpha})$, we assume that its small theta lift $\theta_{\alpha,\beta,\psi}(\pi)$ is not zero.
Then we have
\[
\depth(\pi)=\depth\bigl(\theta_{\alpha,\beta,\psi}(\pi)\bigr).
\]
\end{thm}

\subsection{Depth preserving property: non-quasi-split case}\label{subsec:non-q-spl}
Now we prove the depth preserving property of the local Langlands correspondence for non-quasi-split unitary groups.
Recall that, as we explained in Section \ref{subsec:inner}, every non-quasi-split unitary group is necessarily in even variables.
We put $\G$ and $\G'$ to be the quasi-split unitary groups in $2n$ and $2n+1$ variables over $F$, respectively.

%
%
%

\begin{thm}\label{thm:main}
We assume that the depth preserving property of the local Langlands correspondence for $\G'$ holds.
Namely, for every $L$-parameter $\phi'$ of $\G'$ and every $\pi'\in\Pi_{\phi'}$, we have
\[
\depth(\pi')=\depth(\phi').
\]
Then also the depth preserving property of the local Langlands correspondence for $\G$ holds.
\end{thm}

\begin{rem}\label{rem:prime}
In particular, when the residual characteristic $p$ of $F$ is greater than $2n+1$, this assumption is satisfied by Corollary \ref{cor:odd}.
Here note that the assumption of Corollary \ref{cor:odd} demands that $p$ is greater than $2n+2$.
However, since $p$ is a prime number and $n$ is a positive integer, $p$ is greater than $2n+2$ if and only if $p$ is greater than $2n+1$.
\end{rem}

\begin{proof}
Our task is to prove that, for every $L$-parameter $\phi$ of $\G$ and a member $\pi$ of $\Pi_{\phi}$, we have $\depth(\pi)=\depth(\phi)$.
However, since the local Langlands correspondence for non-tempered representations are constructed by using the theory of Langlands quotients (Theorem \ref{thm:LLC} (4)) and the depth of representations are preserved by taking Langlands quotients, it suffices to consider only the case where $\phi$ is tempered (see \cite[Section 4.3]{Oi:2018a} for details).

Let $\phi$ be a tempered $L$-parameter of $\G$ and $\pi$ be an element of $\Pi_{\phi}$.
Here recall that the Vogan $L$-packet $\Pi_{\phi}$ consists of representations of $\G_{\alpha}(F)$ for a pure inner form $\alpha$ of $\G$.
Let $\alpha$ be the pure inner form of $\G$ such that $\pi$ belongs to $\Pi(\G_{\alpha})$ (namely, $\pi\in\Pi_{\phi}^{\alpha}$).
Then, by Theorem \ref{thm:GI}, there exist a pure inner form $\beta$ of $\G'$ such that the small theta lift $\theta_{\alpha,\beta,\psi}^{\chi,\chi'}(\pi)$ of $\pi$ to $\G'_{\beta}(F)$ with respect to Kudla's splittings belongs to the $L$-packet $\Pi_{\theta(\phi)}^{\beta}$.
Let $\beta$ be such a pure inner form.
We denote the ratio of Kudla's splitting $\G_{\alpha}(F)\rightarrow\widetilde{\G_{\alpha}(F)}$ (resp.\ $\G'_{\beta}(F)\rightarrow\widetilde{\G'_{\beta}(F)}$) to Pan's splitting by $\xi'$ (resp.\ $\xi$).
Then, by Proposition \ref{prop:compare}, the small theta lift of $\pi$ from $\G_{\alpha}(F)$ to $\G'_{\beta}(F)$ with respect to Pan's splittings is given by
\[
\theta_{\alpha,\beta,\psi}^{\chi,\chi'}(\pi\otimes\xi')\otimes\xi^{-1}.
\]
Since the local theta correspondence with respect to Pan's splittings preserves the depth of representations (Theorem \ref{thm:Pan}), we have
\[
\depth(\pi)
=
\depth\bigl(\theta_{\alpha,\beta,\psi}^{\chi,\chi'}(\pi\otimes\xi')\otimes\xi^{-1}\bigr).
\]

We compute the right-hand side.
By the depth preservation of the local Langlands correspondence for odd unitary groups (the assumption of this theorem, see Remark \ref{rem:prime}), we have
\[
\depth\bigl(\theta_{\alpha,\beta,\psi}^{\chi,\chi'}(\pi\otimes\xi')\otimes\xi^{-1}\bigr)
=
\depth\Bigl(\phi_{\theta_{\alpha,\beta,\psi}^{\chi,\chi'}(\pi\otimes\xi')\otimes\xi^{-1}}\Bigr).
\]
Now recall that there exists a character $\xi'_{E^{1}}$ (resp.\ $\xi_{E^{1}}$) of $E^{1}$ satisfying $\xi'=\xi'_{E^{1}}\circ\det$ (resp.\ $\xi=\xi_{E^{1}}\circ\det$).
We denote by $\xi'_{E^{\times}}$ (resp.\ $\xi_{E^{\times}}$) the character of $E^{\times}$ obtained from $\xi'_{E^{1}}$ (resp.\ $\xi_{E^{1}}$) by the isomorphism of Hilbert's theorem 90.
Then, by Proposition \ref{prop:twist} and Theorem \ref{thm:GI}, we have
\begin{align*}
\phi_{\theta_{\alpha,\beta,\psi}^{\chi,\chi'}(\pi\otimes\xi')\otimes\xi^{-1}}
&=
\phi_{\theta_{\alpha,\beta,\psi}^{\chi,\chi'}(\pi\otimes\xi')}\otimes\xi^{-1}_{E^{\times}}\\
&=
\theta(\phi_{\pi\otimes\xi'})\otimes\xi^{-1}_{E^{\times}}\\
&=
\theta(\phi\otimes\xi'_{E^{\times}})\otimes\xi^{-1}_{E^{\times}}\\
&=
\bigl((\phi\otimes\xi'_{E^{\times}}\otimes{\chi'}^{-1}\chi)\oplus\chi\bigr)\otimes\xi_{E^{\times}}^{-1}\\
&=
\phi\otimes(\xi'_{E^{\times}}\chi'^{-1})\otimes(\chi\xi_{E^{\times}}^{-1})\oplus\chi\xi_{E^{\times}}^{-1}.
\end{align*}
By Proposition \ref{prop:depth0}, the depth of characters $\chi\xi_{E^{\times}}^{-1}$ and $\chi'\xi_{E^{\times}}'^{-1}$ are zero.
Namely, the depth of the $L$-parameter 
\[
\phi\otimes(\xi'_{E^{\times}}\chi'^{-1})\otimes(\chi\xi_{E^{\times}}^{-1})\oplus\chi\xi_{E^{\times}}^{-1}
\]
is equal to the depth of $\phi$.
In summary, we get
\[
\depth(\pi)=\depth(\phi).
\]
This completes the proof.
\end{proof}

\end{document}